\documentclass[12pt,reqno]{amsart}
\usepackage{amsthm}
\usepackage{amssymb}
\usepackage{amsmath}
\usepackage{mathrsfs}
\usepackage{amsfonts}
\usepackage{amssymb,amsmath}
 \usepackage[colorlinks, linkcolor=black, citecolor=blue,   hypertexnames=false]{hyperref}
 \usepackage[numbers,sort&compress]{natbib}
\def\R {\mathbb{R}}
\def\N {\mathbb{N}}

\def\d{{\rm d}}

\def \and {{\qquad\text{and}\qquad}}

\newtheorem{proposition}{Proposition}[section]
\newtheorem{theorem}[proposition]{Theorem}

\newtheorem{lemma}[proposition]{Lemma}
\theoremstyle{definition}

\newtheorem{remark}[proposition]{Remark}

\numberwithin{equation}{section}
\def \no#1#2#3 {{\bf #1} (#3), #2.}
\def \eds#1#2#3 {#1, #2, #3.}
\title[Stabilizability and heat equation]
{Characterizations of stabilizable sets  \\for some parabolic  equations in $\R^n$}

\author{Shanlin Huang \quad Gengsheng Wang \quad   Ming Wang}

\address{Shanlin Huang,  School of Mathematics and Statistics, Hubei Key Laboratory of Engineering Modeling and Scientific Computing, Huazhong University of Science and Technology,  Wuhan,  430074,  P.R. China}
\email{shanlin\_huang@hust.edu.cn}
\address{Gengsheng Wang,  Center for Applied Mathematics, Tianjin University, Tianjin 300072, P.R. China}
\email{wanggs62@yeah.net}
\address{Ming Wang,  School of Mathematics and Physics, China University of Geosciences, Wuhan 430074,  P.R. China}
\email{mwang@cug.edu.cn}


\begin{document}

\begin{abstract}

We consider  the parabolic  type equation in $\mathbb{R}^n$:
\begin{align}\label{equ-0}
(\partial_t+H)y(t,x)=0,\,\,\, (t,x)\in (0,\infty)\times \R^n;\;\; \quad y(0,x)\in L^2(\R^n),
\end{align}
where $H$  can be one of the following operators:
$(\romannumeral1)$ a shifted fractional   Laplacian; $(\romannumeral2)$
a shifted Hermite operator; $(\romannumeral3)$ the Schr\"{o}dinger
operator with some  general potentials.
We call a subset $E\subset \mathbb{R}^n$ as a stabilizable set for \eqref{equ-0}, if
there is a  linear bounded operator $K$ on $L^2(\mathbb{R}^n)$
so that the semigroup
$\{e^{-t(H-\chi_EK)}\}_{t\geq 0}$ is exponentially stable.
(Here, $\chi_E$ denotes
the characteristic function of $E$, which is treated
as a linear operator on $L^2(\R^n)$.)

This paper presents different geometric characterizations
of  the stabilizable sets for \eqref{equ-0} with different $H$.
In particular,  when $H$ is a shifted fractional Laplacian,
$E\subset \mathbb{R}^n$ is a stabilizable set for \eqref{equ-0}
 if and only if $E\subset \mathbb{R}^n$ is a thick set, while when $H$ is a
 shifted Hermite operator, $E\subset \mathbb{R}^n$ is
 a stabilizable set for \eqref{equ-0}
 if and only if $E\subset \mathbb{R}^n$ is a set of positive measure.
Our results, together with the results on the observable sets for \eqref{equ-0} obtained in \cite{AB,Ko,Li,M09},
 reveal such phenomena: for some $H$,
the class of  stabilizable sets
contains strictly the class of observable sets,
 while for some other $H$, the classes of
stabilizable sets and observable sets coincide.
Besides, this paper  gives some sufficient conditions on
the
stabilizable sets for \eqref{equ-0} where $H$ is
 the Schr\"{o}dinger operator with some general potentials.

\end{abstract}

\maketitle

%%%%%%%%%%%%%%%%%%%%%%%%%%%%%%%%%%%%%%%%%%%%%%%%%

%%%%%%%%%%%%%%%%%%%%%%%%%%%%%%%%%%%%%%%%%%%%%%%%%
%\section{}
\section{Introduction}\label{section1}

\subsection{Notation}\label{notation}
Let $\N:=\{0,1,2,\dots\}$ and let $\mathbb{N}^+:=\{1,2,\dots\}$.
Write $C(\cdots)$ for a positive constant that depends
on what are enclosed in the brackets.
 Use $\|\cdot\|_{\mathcal{L}(L^2(\R^n))}$
to denote the operator norm on $L^2(\R^n)$.
Use respectively $\|\cdot\|_{L^2(\R^n)}$ and
$(\cdot,\cdot)$ to denote the usual norm and the usual inner product in  $L^2(\R^n)$.
Write respectively $|\cdot|$ and  $\langle\cdot,\cdot\rangle_{\R^n}$ for
the usual norm and the usual inner product  in $\R^n$. Given $x\in \R$,
write $[x]$ for the integer part of $x$.
Given a subset $E\subset\R^n$, write $|E|$ for its Lebesgue measure in $\R^n$
(if it is measurable); write $\chi_E$ for its characteristic function.
Given $L>0$ and $x\in \R^n$, write $Q_L(x)$ for the closed cube (in $\mathbb{R}^n$)
centered at $x$ and of  side-length $L$; write $B(x,L)$
for the closed ball (in $\mathbb{R}^n$) centered at $x$
and of  radius $L$, while use $B^c(x,L)$ to
denote the complement  of $B(x,L)$ in $\R^n$.
Given a function $V$ over $\mathbb{R}^n$, write
$V_{-}(x) := \max\{-V(x),0\}$, $x\in \mathbb{R}^n$.
Given a polynomial $P$, write $\deg P$ for its degree.
Given a linear operator $H$ on a Hilbert space, we write $\sigma(H)$
 for the spectrum  of $H$.
Use $\widehat {\cdot}$ and $\mathcal{F}^{-1}$ to denote
the Fourier transform and its inverse, respectively.
 Write  $(-\Delta)^{\frac{s}{2}}$  (with $s>0$)   for
 the fractional Laplacian  defined  by
$$
(-\Delta)^{\frac{s}{2}}\varphi:=\mathcal{F}^{-1}(|\xi|^{s}\widehat {\varphi}(\xi)),
\,\,\,\varphi\in C_0^\infty(\R^n).
$$

\subsection{Equation}\label{equation}
The subject of this paper is related to the stabilizabilty for the parabolic type equation in $\R^n$:
\begin{align}\label{equ-1}
(\partial_t+H)y(t,x)=0,\,\,\, (t,x)\in (0,\infty)\times \R^n,
\quad y(0,x)\in L^2(\R^n),
\end{align}
where the operator $H$ has one of the following forms:
\begin{itemize}
  \item [(i)] The first form is as:
  \begin{align}\label{equ-frac}
H=(-\Delta)^{\frac{s}{2}}-c,
\end{align}
where $s>0$ and $c\in\mathbb{R}$.
 \item [(ii)] The second form is as:
 \begin{align}\label{equ-herm}
H=-\Delta+|x|^2-c,
\end{align}
where $c\in \mathbb{R}$.

  \item [(iii)] The third form is as:
   \begin{align}\label{1.4,5-27w}
   H=-\Delta +V(x),
   \end{align}
   where the real-valued potential $V$ satisfies one of the following two conditions:

   \vskip 5pt

\noindent{\it Condition I} $\;$  The function $V$ is locally integrable
so that for some $\delta\in (0,1)$,
$$
\int_{\R^n}V_{-}(x)|\varphi|^2\d x\leq \delta \int_{\R^n}|\nabla \varphi|^2\d x,
\;\mbox{when}\; \varphi\in C_0^\infty(\R^n).
$$
\vskip 5pt
\noindent{\it Condition II} $\;$ The function $V$ is
locally bounded and  measurable  so that
$$
\lim_{|x|\to \infty} V(x) =\infty.
$$
\end{itemize}
Several notes on the above equation are given in order.

\begin{itemize}

\item [($\textbf{a}_1 $)] The form $H$ given by \eqref{1.4,5-27w} is a
generalization of  that given by  \eqref{equ-herm}. The reasons that
we consider them as two different
    cases are as follows: First, our results for \eqref{equ-1} with
    \eqref{equ-herm} are much more delicate than those
    for \eqref{equ-1} with \eqref{1.4,5-27w}; Second, our
     methods to study them are totally different.

  \item [($\textbf{a}_2$)]  We call \eqref{equ-1} with \eqref{equ-frac}
a  shifted fractional heat equation (a fractional heat equation, for short).
It is well known that
in the case  \eqref{equ-frac},
 $(-H)$ is self-adjoint and generates an analytic semigroup
  $\{e^{-tH}\}_{t\geq 0}$  satisfying
\begin{align}\label{1.9-2-28-w}
\left\|e^{-tH}\right\|_{\mathcal{L}(L^2(\R^n))}=e^{ct},\,\,\, \mbox{ when }\, t\geq 0.
\end{align}

\item [($\textbf{a}_3 $)] We call \eqref{equ-1} with \eqref{equ-herm}
a heat equation associated with a shifted Hermite operator.
It follows by \cite{Ti} that
    in the case \eqref{equ-herm},   $(-H)$
    is self-adjoint and generates an analytic semigroup $\{e^{-tH}\}_{t\geq 0}$  satisfying
 \begin{align}\label{1.10,5.28w}
\left\|e^{-tH}\right\|_{\mathcal{L}(L^2(\R^n))}=e^{(c-n)t},\,\,\, \mbox{when}\,\, t\ge 0.
\end{align}

\item [($\textbf{a}_4 $)] We call \eqref{equ-1} with \eqref{1.4,5-27w}
a heat equation with a potential. The Schr\"{o}dinger operator with potentials satisfying either
{\it Condition I} or {\it Condition II} is important and has been widely studied.
When $V$ satisfies {\it Condition I},  $(-H)$
 is self-adjoint  and generates an analytic semigroup (see,
e.g.,  \cite[Theorem \uppercase\expandafter{\romannumeral10}.17]{RS}).
 When $V$ satisfies {\it Condition II},
$(-H)$ is  self-adjoint, generates an analytic semigroup and has
a discrete spectrum (see \cite[Theorem 3.1]{BS}).

\item [($\textbf{a}_5 $)] Let $n\ge3$ and let
\begin{eqnarray*}
V(x):=-\Big(\frac{n-2}{2}\Big)^{2}\frac{\delta}{|x|^2},\;\;x\in \mathbb{R}^n,
\end{eqnarray*}
where $\delta\in (0,1)$. Then $V$
 satisfies  {\it Condition I}.
 This can be derived
 from the classical Hardy inequality:
$$
\int_{\mathbb{R}^{n}}|\nabla \varphi(x)|^{2} \d x
\geq\Big(\frac{n-2}{2}\Big)^{2} \int_{\mathbb{R}^{n}}
\frac{|\varphi(x)|^{2}}{|x|^{2}} \d x, \quad  \varphi\in C_0^\infty(\R^n ).
$$
For the above $V$,
the corresponding equation   \eqref{equ-1} with \eqref{1.4,5-27w}
is not exponentially stable\footnote{The equation \eqref{equ-1} is said to be exponentially
stable,
if there is $M>0$ and $\omega>0$ so that
$\|e^{-tH}\|_{\mathcal{L}(L^2(\R^n))}\leq Me^{-\omega t}$
for all $t\geq 0$.}.
Indeed, by \cite[Theorem 4.1]{BS}, we have
$$
\left\|e^{-tH}\right\|_{\mathcal{L}(L^2(\R^n))}=1\;\; \mbox{for each}\;\, t\geq 0.
$$

\item [($\textbf{a}_6 $)] Let $a>0$ and $c\in \mathbb{R}$, and let
\begin{eqnarray}\label{1.8-5-28wangw}
V(x):=|x|^a-c,\;\;x\in \mathbb{R}^n.
\end{eqnarray}
 Clearly, the above $V$   satisfies {\it Condition II}.
For this  $V$, when $c\geq \lambda_1$ (the first eigenvalue
of the  operator $-\Delta+|x|^a$),
the corresponding equation \eqref{equ-1} with \eqref{1.4,5-27w}
is not exponentially  stable.

\end{itemize}

\subsection{Concepts}\label{concept}
To give our main results, we need the following concepts
about the equation \eqref{equ-1}:
\begin{itemize}
\item [($\textbf{b}_1$)]
A measurable set $E\subset \mathbb{R}^n$ is called a
 {\it stabilizable set}  for \eqref{equ-1},
 if  there is a linear  bounded
operator  $K$ on $L^2(\R^n)$ so that
for some constants $M>0$ and $\omega>0$,
\begin{align}\label{equ-stable-1}
\left\|e^{-t(H-\chi_E K)}\right\|_{\mathcal{L}(L^2(\R^n))}\leq Me^{-\omega t}
\;\;\mbox{for all}\;\; t\geq 0.
\end{align}
 When $E\subset \mathbb{R}^n$ is  a
 {\it stabilizable set}  for \eqref{equ-1},
 we say that the equation \eqref{equ-1} is
 stabilizable over $E\subset \mathbb{R}^n$.

\item [($\textbf{b}_2$)] A measurable set
$E\subset \mathbb{R}^n$ is called a thick set, if
there is $\gamma>0$
 and $L>0$ so that
\begin{align}\label{equ-set-thick}
\left|E \bigcap Q_L(x)\right|\geq \gamma L^n\;\;\mbox{for each}
\;\;x\in \mathbb{R}^n.
\end{align}
\end{itemize}

Some notes on the above concepts are given in order.
  \begin{itemize}

\item [($\textbf{c}_1$)]
The concept of {\it stabilizable sets} seems to be new for us. It
links with the stabilizability and is
comparable to the concept of {\it observable sets} given in \cite{WWZZ},
i.e.,  a measurable set $E\subset \mathbb{R}^n$ is called an
  {\it observable set} for \eqref{equ-1},  if for every $T>0,$
 there is $C=C(E,T)>0$ so that when $y$ solves \eqref{equ-1},
\begin{align}\label{equ-ob}
 \|y(T,\cdot)\|_{L^2(\R^n)} \leq C\left(\int_{0}^T
 \int_{E}|y(t,x)|^2\d x\d t\right)^{1/2}.
\end{align}

\item [($\textbf{c}_2$)]
 There have been several common methods to study the stabilizability for
  control systems, such as LQ theory and
 Lyapunov functions (see, for instance, \cite{coron, LT, yashan, Z}).
 We would like to mention the recent work
  \cite{TWX} which gives
  a new characterization of the stabilizability  in terms of a weak
  observability inequality (see Theorem \ref{thm-txw} in the current paper).
  It plays an important role in our studies.

\item [($\textbf{c}_3$)] To our best knowledge, the concept of \emph{thick sets}  arose from studies of the uncertainty principle (see, for instance, \cite[p. 5]{BD}, or \cite[p. 113]{HJ})

    \end{itemize}

\subsection{Aim and motivation}\label{aim and motivation}
First, the concept of {\it stabilizable sets} connects with the stabilizability, while the concept
of {\it observable sets} links with the null controllability. Since the null controllability
implies the stabilizabilty (see, e.g., \cite[p. 227]{Z}), we see that if $E$ is an  observable set
for \eqref{equ-1}, then it is a {\it stabilizable set } for \eqref{equ-1}.
Second, the {\it observable sets} for heat equations have been studied
in, for instance,
\cite{AB,MZ01a,BJP,DWZ,MZ01b,MZ,Ko,LM19,Li,Mi05,Mi06,M09,Lu,DM,PW,AEWZ,EV18,trelatZ,WWZZ}.
Especially, according to \cite[Theorem 1.1]{WWZZ} (see also \cite{EV18}),
the {\it observable sets} for \eqref{equ-1}, where $H=-\Delta$,
are characterized by thick sets, while  according to \cite[Remark 1.13]{AB}, the {\it observable sets} for
\eqref{equ-1}, with \eqref{equ-frac}
where $s>1$ and $c\in\R$, are also characterized by thick sets.
 Thus, we naturally ask
the following question:
\vskip 5pt

\noindent{\it How to characterize  {\it stabilizable sets}
  for  \eqref{equ-1}?}

  \vskip 5pt

The answer for the above question may give
geometric characterizations of {\it stabilizable sets}
for \eqref{equ-1}.
Such characterizations, along with the characterizations
on observable sets obtained in \cite{AB,Ko,Li,M09},
may
give an explicit gap between the stabilizability and the null
controllability from the perspective of control regions.
The studies on the above problem may  lead us to a new subject about
   the stabilizability of PDEs.
To our best knowledge,  the above problem
 has not been touched upon.

The aim of this  paper is to answer  the above  question.

\subsection{Main results}\label{main results}
The first main result concerns \eqref{equ-1} with \eqref{equ-frac}.

\begin{theorem}\label{thm-frac-laplace}
Let $E\subset\R^n$. Then the following statements  are equivalent:
\begin{itemize}
  \item [(i)] The set $E$ is a thick set.
  \item [(ii)] The set $E$ is a {\it  stabilizable set}
  for \eqref{equ-1} with \eqref{equ-frac} where $s>0$ and
  $c\geq 0$.

\end{itemize}

\end{theorem}
Several notes on Theorem \ref{thm-frac-laplace} are given in order.

\begin{itemize}

\item [($\textbf{d}_1$)] We explain why there is
the restriction $c\geq 0$ in $(ii)$ of Theorem \ref{thm-frac-laplace}:
From \eqref{1.9-2-28-w}, we see that when $c<0$,
the equation \eqref{equ-1}, with  \eqref{equ-frac} where $s>0$,
is  exponentially stable,
while when $c\geq 0$, it  is not.
Thus, when studying the stabilization for \eqref{equ-1}
with \eqref{equ-frac}, we only need to focus on the case
 that $c\geq 0$ and $s>0$.

\item [($\textbf{d}_2$)]
Since each  {\it observable set} is a {\it stabilizable set},
  the class of  {\it stabilizable sets} contains
  the class of  {\it observable sets}.
  On the other hand,
  for \eqref{equ-1} with \eqref{equ-frac} where $s>1$ and $c\in \mathbb{R}$,
    it was proved in \cite[Remark 1.13]{AB} that
  the {\it observable sets} are characterized by thick sets.
   Then by
   Theorem  \ref{thm-frac-laplace}, we see that
    the classes
   of {\it stabilizable sets} and {\it observable  sets}
   coincide for this case.

\item [($\textbf{d}_3$)] In the case that $s\in (0,1)$ and  $c\geq 0$, it follows from Theorem \ref{thm-frac-laplace}
that the {\it stabilizable sets} for \eqref{equ-1} with \eqref{equ-frac} are characterized by thick sets,
while in the case that either $s\in (0,1)$, $c\in \mathbb{R}$ and $n\geq 1$ or $s=1$,
$c\in \mathbb{R}$ and $n=1$,
it follows from \cite[Theorem 3 \& Remark 7]{Ko} and
\cite[Theorem 1.1]{Li} that
 some thick sets (for instance, $B^c(x,R)$ with
 $x\in \R^n$ and $R>0$) are not {\it observable sets} for
the equation \eqref{equ-1} with  \eqref{equ-frac}.
Hence, for the equation \eqref{equ-1} with \eqref{equ-frac} where either
 $s\in (0,1)$, $c\geq 0$ and $n\geq 1$ or $s=1$, $c\geq 0$ and $n=1$,
the class of {\it stabilizable sets} contains strictly the class of
{\it observable sets}. We would like to mention what follows:
For the case where $s=1$, $c\in \mathbb{R}$ and $n\ge 2$,
it is still open whether there
is a thick set which is not an {\it observable set} for \eqref{equ-1}
 with \eqref{equ-frac},  to our best knowledge   (see \cite{Ko,Li}).

\item  [($\textbf{d}_4$)]
To prove  $(i)\Rightarrow (ii)$, we built up  an
abstract criteria for the stabilizability, i.e., Lemma \ref{lem-0408-01}. It
 can be viewed as
a variant of Lebeau-Robbiano strategy and has independent significance.
\end{itemize}

\vskip 5pt

The second main result concerns with  the equation
\eqref{equ-1} with  \eqref{equ-herm}.

\begin{theorem}\label{thm-hermite}
Let $E\subset \R^n$.
 Then the following statements are equivalent:
\begin{itemize}
  \item [(i)]  The set $E$ has a positive  measure.
  \item [(ii)] The set $E$ is a {\it stabilizable set}
  for \eqref{equ-1} with \eqref{equ-herm} where $c\geq n$.
\end{itemize}
\end{theorem}
Some remarks about Theorem \ref{thm-hermite}
are given as follows:
\begin{itemize}
\item [($\textbf{e}_1$)] We explain why there is the restriction
$c\geq n$ in $(ii)$ of Theorem \ref{thm-hermite}: From  \eqref{1.10,5.28w}, we see that
    when $c<n$, the equation \eqref{equ-1} with \eqref{equ-herm}
     is exponentially stable,
    while when $c\geq n$, it is not.
    Thus, when studying the stabilization for \eqref{equ-1}
    with \eqref{equ-herm}, we only need to focus on the case that $c\geq n$.

\item [($\textbf{e}_2$)] For the equation \eqref{equ-1}
with \eqref{equ-herm} where $c\geq n$,
the class of {\it observable sets} is strictly contained
in the class of {\it  stabilizable sets}.
Indeed, given a set of positive measure
in a half space of $\R^n$, we see from
\cite[Theorem 1.10]{M09} that it is not an {\it observable set},
 while we find by
Theorem \ref{thm-hermite} that it is a {\it stabilizable set}.

\item [($\textbf{e}_3$)] From Theorem \ref{thm-frac-laplace}
and Theorem \ref{thm-hermite}, we see that
the geometric characterizations for \eqref{equ-1}, with
\eqref{equ-frac} and \eqref{equ-herm} respectively, are
different. The reason is that
 the spectral inequalities for $H$, given by \eqref{equ-frac}
and \eqref{equ-herm} respectively,
are  different (see Lemma \ref{lemm-spec-frac}
and Lemma \ref{lemm-spec-herm}).

\end{itemize}

\subsection{Some sufficient conditions on  stabilizable sets}

The following two  results concern
sufficient conditions on  {\it  stabilizable sets} for
the equation \eqref{equ-1} with \eqref{1.4,5-27w}.

\begin{theorem}\label{thm-general-1}
Suppose that $V$ satisfies {\it Condition I}. Then every
thick set is a {\it  stabilizable set} for \eqref{equ-1} with \eqref{1.4,5-27w}.
\end{theorem}

\begin{theorem}\label{thm-general-2}
Suppose that $V$ satisfies {\it Condition II}. Then each nonempty
open set is a {\it  stabilizable set} for \eqref{equ-1} with \eqref{1.4,5-27w}.
\end{theorem}
Some notes on Theorem \ref{thm-general-1} and  Theorem \ref{thm-general-2} are given in order.
\begin{itemize}
\item [($\textbf{f}_1$)] From ($\textbf{a}_5$) and ($\textbf{a}_6$)
in Subsection \ref{equation}, we see that when $V$ verifies either
{\it Condition I} or {\it Condition II},
     the equation \eqref{equ-1} with \eqref{1.4,5-27w}
     may not be exponentially stable. Thus,
    it makes sense to study the stabilization
    for \eqref{equ-1} with \eqref{1.4,5-27w}.

  \item [($\textbf{f}_2$)] We cannot get a sufficient
  and necessary condition on the {\it stabilizable sets} for \eqref{equ-1}
  with \eqref{1.4,5-27w}, since
  we are not able to get the desired spectral inequality for
the operator $H$  given by \eqref{1.4,5-27w}. Thus we
cannot use the above-mentioned Lemma \ref{lem-0408-01}
    to prove  either Theorem \ref{thm-general-1} or
    Theorem \ref{thm-general-2}.

\item [($\textbf{f}_3$)]
We use  different approaches to show Theorem \ref{thm-general-1} and
    Theorem \ref{thm-general-2}:
    When $V$ satisfies
{\it Condition I}, we can treat $V$ as a small perturbation of $-\Delta$
in the sense of the quadratic form. Thus, we only need
the spectral inequality for the Laplacian in the
 proof of Theorem \ref{thm-general-1};
 When   $V$ satisfies
{\it Condition II}, we can use
a unique continuation property of the elliptic equation $Hu=0$,
as well as the techniques used in \cite{BT} (see also in \cite{B}), to prove
  Theorem \ref{thm-general-2}.

\item [($\textbf{f}_4$)]  As a comparison, we mention some sufficient
 conditions on the {\it observable sets} for \eqref{equ-1}
  with \eqref{1.4,5-27w}. First, for the case that
  $$
  V(x)=|x|^{2k},\;\;x\in\R^n,\;\;\mbox{with}\;\;k\in \N^+,
  $$
     it was obtained in  \cite{M09,DM} that when $k\geq 2$,
the cone
$$
E:=\{x\in \mathbb{R}^{n}: |x|\geq r_0, x/|x|\in \Theta_0\}
$$
(Here, $r_0> 0$ and  $\Theta_0$
 is a nonempty open subset of $\mathbb{S}^{n-1}$.)
 is an {\it observable set}, while when $k=1$, the above cone is no longer
 an {\it observable set};  it was proved in  \cite{BJP} that when $k=1$, every thick set
 is an {\it observable set}. Second, for the case that $V$ is a real valued analytic
 potential vanishing at infinity,
it was proved in  \cite{LM19} that each thick set is an {\it observable set}.

Besides, we would like to mention the recent work \cite{DWZ} where {\it observable sets} were studied for
\eqref{equ-1}
  with \eqref{1.4,5-27w} where potentials are  bounded and time-dependent.
\end{itemize}

\subsection{Plan of the paper}\label{plan}

The  rest of the paper is organized as follows:  Section \ref{sec-abstract}
presents an abstract criteria for
the stabilization. Section \ref{sec-2.2} gives two spectral inequalities.
Section \ref{sec-3} proves Theorem \ref{thm-frac-laplace}
and Theorem \ref{thm-hermite} with the aid of lemmas built up
in Section \ref{sec-abstract} and Section \ref{sec-2.2}.
Section \ref{sec-4} shows Theorem \ref{thm-general-1}
and Theorem \ref{thm-general-2} via different approaches.

\section{An abstract criteria for stabilization}\label{sec-abstract}

This section presents a criteria for the stabilizability from
the perspective of spectral inequalities.
First of all, we introduce the next Theorem \ref{thm-txw} which
is a direct consequence of  \cite[Theorem 1]{TWX} and plays an important role in our studies.

\begin{theorem}[\cite{TWX}]\label{thm-txw}
Let $E$ be a measurable subset  of $\R^n$. Then
$E$ is a {\it
 stabilizable set} for \eqref{equ-1}  if and only if   there is $T>0$, $\alpha\in (0,1)$ and  $C=C(E,T)>0$ so that when $y$
 solves \eqref{equ-1},
\begin{align}\label{equ-stable}
 \|y(T,\cdot)\|_{L^2(\R^n)} \leq C\left(\int_{0}^T
 \int_{E}|y(t,x)|^2\d x\d t\right)^{1/2} + \alpha  \|y(0,\cdot)\|_{L^2(\R^n)}.
\end{align}
\end{theorem}
 We call \eqref{equ-stable} a weak observability inequality.
Before stating the above-mentioned criteria, we explain its main connotations:
The Lebeau-Robbiano strategy (see \cite{LR}) leads to the observability inequality,
 equivalently the null controllability,
 for the heat equation.
This strategy has been generalized to abstract settings in  some Hilbert spaces
(see \cite{BP,M09,M10,TT}) and   in some Banach spaces (see \cite{GST}).
In essence, it  is  a combination of the   Lebeau-Robbiano spectral inequality (see \cite{LR})
and a dissipative inequality.
 A key condition in this combination is as: the decay rate in
 the  dissipative inequality is strictly larger than the growth rate
 in the spectral inequality.
 Since the  gap between the null controllability and the stabilizabilty is explicitly given by
\eqref{equ-ob} and \eqref{equ-stable},
we find
that the aforementioned key condition can be relaxed in
the studies of the stabilizabilty. This relaxed condition leads
to an abstract criteria for the stabilizabilty in
the next Lemma \ref{lem-0408-01}. It plays an important role
in the  proofs of both Theorem \ref{thm-frac-laplace}
and Theorem \ref{thm-hermite}.

\begin{lemma}\label{lem-0408-01}
Let $H$ be a self-adjoint operator so that $(-H)$
 generates a $C_0$ semigroup $\{e^{-tH}\}_{t\geq 0}$
on $L^2(\R^n)$. Let $E$ be a measurable
set in $\R^n$ and let $\{\pi_{k}\}_{k\in \N^+}$ be a family
of orthogonal projections  on $L^2(\R^n)$.
If there are positive constants   $c_1, c_2, a, b,M$ so that  for each $k\in\N^+$,
the   spectral inequality
 \begin{equation}\label{equ-0408-05}
\|\pi_k \varphi\|_{L^2(\R^n)}\leq e^{c_1k^a}\|\pi_k \varphi\|_{L^2(E)},
\; \mbox{when}\;\varphi\in L^2(\R^n)
\end{equation}
and the   dissipative inequality
 \begin{equation}\label{equ-0408-06}
\left\|(1-\pi_k) (e^{-tH}\varphi)\right\|_{L^2(\R^n)}
\leq M e^{-c_2tk^b}\|\varphi\|_{L^2(\R^n)},
\;\mbox{when}\;\varphi\in L^2(\R^n),\,t>0
\end{equation}
hold,
then $E$ is a  stabilizable set for \eqref{equ-1}.
\end{lemma}

\begin{remark}\label{remark2.2,6-10wang}
 In Lemma \ref{lem-0408-01}, the family
$\{\pi_{k}\}_{k\in \N^+}$ can be replaced by a family
$\{\pi_{k}\}_{1\leq k\in \R}$ of orthogonal projections  on $L^2(\R^n)$.
\end{remark}

\begin{remark}\label{rmk-le-Ro}
 We  compare Lemma \ref{lem-0408-01}
 with the abstract Lebeau-Robbiano strategy
in the following manner:
First, the abstract Lebeau-Robbiano strategy says that
(see e.g. in \cite[Theorem 2.2]{M10}, as well as \cite[Theorem 2.1]{BP})
the observability inequality \eqref{equ-ob} holds, if
\begin{align}\label{equ-0408-06-1-1}
\mbox{\eqref{equ-0408-05} and \eqref{equ-0408-06} are true, and}\,\, a<b.
\end{align}
In \eqref{equ-0408-06-1-1}, the condition $a<b$ is necessary.
 Indeed, there are examples showing that if \eqref{equ-0408-05}
  and \eqref{equ-0408-06} hold
for some $E\subset \R^n$, but $a\geq b$, then
the observability inequality \eqref{equ-ob} is not true  for any $T>0$
(see Remark \ref{rmk-0430-1} for details).
Second,   Lemma \ref{lem-0408-01} says: to get
\eqref{equ-stable},
we only need \eqref{equ-0408-05} and \eqref{equ-0408-06}, but not
the condition $a<b$.
  The advantage we can take from this is as follows:
  We are allowed to afford more cost in the spectral inequality
  when considering smaller $E$. Thus,
  the class of {\it stabilizable sets}  might contain strictly
  the class of {\it observable sets}. These will be discussed in detail
  in      Subsection \ref{subsec-3.3}.

\end{remark}

\begin{proof}[Proof of Lemma \ref{lem-0408-01}.]
According to  Theorem \ref{thm-txw}, to prove that $E$ is a {\it stabilizable set} for \eqref{equ-1},
   it suffices to show
\begin{align}\label{equ-513-1}
\left\|e^{-TH}\varphi\right\|_{L^2(\R^n)}\leq C\Big(\int_0^T\left\|
e^{-tH}\varphi\right\|^2_{L^2(E)}\d t\Big)^{1/2} + \alpha \|\varphi\|_{L^2(\R^n)},
 \;\mbox{when}\; \varphi\in  L^2(\R^n)
\end{align}
for some $C>0$, $T>0$ and $\alpha\in (0,1)$.
Before proving \eqref{equ-513-1}, we give some preliminaries. First,
since $(-H)$ generates a $C_0$ semigroup in $L^2(\R^n)$,
there  is $\delta_0\ge 0$ and $M'>0$  so that
$$
\left\|e^{-tH}\right\|_{\mathcal{L}(L^2(\R^n))}\leq M'e^{\delta_0 t},
\quad  t\geq 0.
$$
Without loss of generality, we can assume that $M'\leq M$
where $M$ is given by \eqref{equ-0408-06}. Then the
  shifted operator\footnote{Here $I$ is the identity operator on $L^2(\mathbb{R}^n)$.}
\begin{equation}\label{equ-0512-002}
\tilde{H}=H+\delta_0I
\end{equation}
 satisfies
\begin{equation}\label{equ-0512-003}
\big\|e^{-t\tilde{H}}\big\|_{\mathcal{L}(L^2(\R^n))}\leq M,\quad t\geq 0.
\end{equation}
Second, with $a,b,c_1,c_2,M$ given in
\eqref{equ-0408-05}-\eqref{equ-0408-06}, we set
\begin{align}\label{equ-513-2}
  \gamma:=2^{\frac{a}{b}+a}>1,
  \quad N:=\max \Big\{2,\, \frac{2^{b+2}\delta_0}{c_2}\Big\},
\end{align}
\begin{align}\label{equ-513-3}
 C(M, \gamma):= M^2 +\frac{\gamma-1}{8M^2}
 \Big(\frac{4M^4}{\gamma}\Big)^{\gamma/(\gamma-1)},\quad D(M, N)
 :=\frac{e^{-2c_1(2N)^{a/b}}}{8M^2N},
\end{align}
\begin{align}\label{equ-513-4}
A:=\frac{2^{b+1}}{c_2}\ln{
\left(1+\frac{25C(M, \gamma)}{D(M, N)}\right)},
\quad \tau_0:=\frac{3A}{2N}.
\end{align}
Third, we define a function:
\begin{align}\label{equ-513-5}
g(\tau):=\frac{\tau}{4M^2}\exp\left\{-2c_1
\left[({A}/{\tau})^{\frac{1}{b}}\right]^a\right\},
\quad \tau\in (0,\tau_0).
\end{align}
{\it Recall that for each $x\in\R$, $[x]$ denotes the integer part of $x$. }

We now prove \eqref{equ-513-1}
by two steps\footnote{Some ideas are borrowed from \cite[Theorem 2.1]{BP}, see also
\cite[Theorem 2.2]{M10}.}:

\vskip 5pt

\noindent \emph{Step 1.  We prove the following  recurrence inequality:
when $\tau\in (0, \tau_0)$ and $\varphi\in L^2(\R^n)$,
\begin{equation}\label{equ-0408-07}
g(\tau)\big\|e^{-\tau \tilde{H}}\varphi\big\|^2_{L^2(\R^n)}
-g({\tau}/{2})\|\varphi\|^2_{L^2(\R^n)}
\leq \int_{\frac{\tau}{2}}^{\tau}{\big\|e^{-t\tilde{H}}\varphi\big\|^2_{L^2(E)}\d t}
+\alpha_0\tau\|\varphi\|^2_{L^2(\R^n)},
\end{equation}
where
\begin{align}\label{equ-513-6}
\alpha_0:=D(M, N)\exp\{- c_2 2^{-(b+1)}A\}/50.
\end{align}
}

First, we observe that when
  $k\in \N^+$,  $\varphi\in L^2(\R^n)$ and $t\in(0,\tau_0)$,
\begin{align}\label{equ-0408-10}
\frac{1}{2}e^{-2c_1k^a}\|e^{-t \tilde{H}}\varphi\|^2_{L^2(\R^n)}
&= \frac{1}{2}e^{-2c_1k^a}\left(\|\pi_k e^{-t\tilde{H}}\varphi\|^2_{L^2(\R^n)}
+\|(1-\pi_k)e^{-t\tilde{H}}
\varphi\|^2_{L^2(\R^n)}\right) \nonumber\\
 &\leq  \frac12 \|\pi_k e^{-t\tilde{H}}\varphi\|^2_{L^2(E)}
 +\frac12e^{-2c_1k^a}\|(1-\pi_k)e^{-t\tilde{H}}\varphi\|^2_{L^2(\R^n)}
       \nonumber\\
  &\leq  \|e^{-t\tilde{H}}\varphi\|^2_{L^2(E)}
  +(1+\frac12e^{-2c_1k^a})\|(1-\pi_k)
  e^{-t\tilde{H}}\varphi\|^2_{L^2(\R^n)}
     \nonumber\\
   &\leq    \|e^{-t\tilde{H}}\varphi\|^2_{L^2(E)}
   + M^2(1+\frac12e^{-2c_1k^a})e^{-2c_2tk^b} \|\varphi\|^2_{L^2(\R^n)}.
\end{align}

In \eqref{equ-0408-10}, we used the fact that each $\pi_{k}$ is an
orthogonal projection  on $L^2(\R^n)$ on Line 1;
 we used the spectral inequality \eqref{equ-0408-05}
on Line 2; we used the Cauchy Schwartz inequality
and the fact that $\|\varphi\|_{L^2(E)}\leq \|\varphi\|_{L^2(\R^n)}$ on Line 3; we used the dissipative
inequality \eqref{equ-0408-06}
and the fact that $e^{-\delta_0t}\le 1$
when $0<t<\tau_0$ on Line 4.

 Next, we arbitrarily fix $\tau\in(0,\tau_0)$ and $\varphi\in L^2(\R^n)$.
 Integrating both sides of \eqref{equ-0408-10} with respect to $t$ on  $(\frac{\tau}{2}, \tau)$,
  using \eqref{equ-0512-003}, we find
  that for all $k\in \N^+$,
\begin{align}\label{equ-0408-11}
& \frac{\tau}{4M^2}e^{-2c_1k^a}
\big\|e^{-\tau \tilde{H}}\varphi\big\|^2_{L^2(\R^n)}
\leq\frac{1}{2}e^{-2c_1k^a}\int_{\frac{\tau}{2}}^{\tau}
{\big\|e^{-t\tilde{H}}\varphi\big\|^2_{L^2(\R^n)}\,\d t} \nonumber\\
 &\leq  \int_{\frac{\tau}{2}}^{\tau}{ \big\|e^{-t\tilde{H}}\varphi\big\|^2_{L^2(\R^n)}\,\d t} +M^2\big(1+\frac12e^{-2c_1k^a}\big)\int_{\frac{\tau}{2}}^{\tau}{e^{-2c_2tk^b} \|
 \varphi\|^2_{L^2(\R^n)}\,\d t}    \nonumber\\
  &\leq   \int_{\frac{\tau}{2}}^{\tau}
  { \big\|e^{-t\tilde{H}}\varphi\big\|^2_{L^2(\R^n)}\,\d t}
  +M^2\big(1+\frac12e^{-2c_1k^a}\big)\frac{\tau e^{-c_2\tau k^b}}{2}\|\varphi\|^2_{L^2(\R^n)}.
\end{align}
In  the last inequality of \eqref{equ-0408-11}, we  used the  estimate:
$$
\int_{\frac{\tau}{2}}^{\tau}{e^{-2c_2tk^b}\,\d t}
<\frac{\tau}{2}e^{-c_2\tau k^b}.
$$
Set
\begin{align}\label{equ-513-7}
  k(\tau):=\left[({A}/{\tau})^{\frac{1}{b}}\right].
\end{align}
By \eqref{equ-513-5} and \eqref{equ-513-7}, it follows that
\begin{align}\label{equ-0408-12}
\frac{\tau}{4M^2}e^{-2c_1k(\tau)^a}=g(\tau).
\end{align}
Meanwhile, by    \eqref{equ-513-2}, \eqref{equ-513-4}
and \eqref{equ-513-7}, it follows  that $k(\tau)\in \N^+$.
Thus, we have \eqref{equ-0408-11} where $k=k(\tau)$.

 We now claim
 \begin{align}\label{equ-0408-13}
M^2\Big(1+\frac12e^{-2c_1k(\tau)^a}\Big)
\frac{\tau e^{-c_2\tau k(\tau)^b}}{2}\leq g({\tau}/{2})+\alpha_0\tau.
\end{align}
When \eqref{equ-0408-13} is proved, the desired \eqref{equ-0408-07}
follows from \eqref{equ-0408-11}
(with $k=k(\tau)$), \eqref{equ-0408-12}  and \eqref{equ-0408-13} at once.

The rest of this step is to show \eqref{equ-0408-13}.
For this purpose,
  we first claim
  \begin{align}\label{equ-0408-13.5}
M^2B\left(1+\frac{x}{2}\right)\leq \frac{1}{8M^2}x^{\gamma}
+\alpha_0,\;\;\mbox{when} \,\,x\in (0,1),
\end{align}
where
\begin{align}\label{2.18-5-29wang}
 B:=\frac{e^{-c_22^{-b}A}}{2}\in (0, 1).
  \end{align}
    To this end, we define a function:
  \begin{eqnarray*}
  F(x):=\frac{1}{8M^2}x^{\gamma}-\frac{M^2B}{2}x
  +\alpha_0-M^2B,\;\; x>0.
  \end{eqnarray*}
  A direct computation shows
  \begin{eqnarray*}
  F'(x)<0,\;\;\mbox{when}\;\;x\in (0,x_0);\;\;
  F'(x)>0,\;\;\mbox{when}\;\;x\in (x_0,\infty),
  \end{eqnarray*}
  where
  \begin{eqnarray*}
 x_0:=\Big(\frac{4M^4B}{\gamma}\Big)^{\frac{1}{\gamma-1}}.
  \end{eqnarray*}
  Then we have
  \begin{align}\label{equ-513-8}
\min_{x>0}F(x) &= F(x_0)\nonumber\\
&=\alpha_0-M^2B-\frac{(\gamma-1)}{8M^2}
\Big(\frac{4M^4B}{\gamma}\Big)^{\frac{\gamma}{\gamma-1}}\nonumber\\
 &\geq  \alpha_0-\Big(M^2+\frac{\gamma-1}{8M^2}
 \Big(\frac{4M^4}{\gamma}\Big)^{\gamma/(\gamma-1)}\Big)B \nonumber\\
  &>0.
\end{align}
In \eqref{equ-513-8}, for the first inequality, we used the fact that
$B^{\frac{\gamma}{\gamma-1}}<B$ (which follows from
facts that $B\in (0, 1)$ and $\gamma/(\gamma-1)>1$);
for the last inequality, we first see from \eqref{equ-513-6}
and \eqref{2.18-5-29wang} that it
is equivalent to
\begin{eqnarray}\label{2.21-5-29wang}
A>\frac{2^{b+1}}{c_2}\ln \frac{25\big(M^2+\frac{\gamma-1}{8M^2}
(\frac{4M^4}{\gamma})^{\gamma/(\gamma-1)}\big)}{D(M, N)},
\end{eqnarray}
then we obtained \eqref{2.21-5-29wang}
from
 \eqref{equ-513-3} and \eqref{equ-513-4}.
 Now \eqref{equ-0408-13.5} follows from \eqref{equ-513-8} at once.

 We next use \eqref{equ-0408-13.5} to show \eqref{equ-0408-13}.
 Indeed, since
 $$
({A}/{\tau})^{\frac{1}{b}}
<2\big[({A}/{\tau})^{\frac{1}{b}}\big]=2k(\tau),
$$
we find that
\begin{align*}
k(\tau)^b>\frac{A}{2^b\tau}\,\,\,\,\mbox{and}
\,\,\,\,2^{\frac1b+1}k(\tau)>\Big(\frac{2A}{\tau}\Big)^{\frac{1}{b}}
>k({\tau}/{2}),
\end{align*}
 which yields
 \begin{align*}
e^{-c_2\tau k(\tau)^b}<e^{-c_22^{-b}A}\,\,\,\,
\mbox{and}\,\,\,\,e^{-2c_1k(\tau)^a\gamma}<e^{-2c_1k({\tau}/{2})^a}.
\end{align*}
These, along with \eqref{equ-0408-13.5}
(where  $x:=e^{-2c_1k(\tau)^{a}}\in (0, 1)$) and \eqref{equ-0408-12},
lead to
 \eqref{equ-0408-13}.

 Hence the proof of \eqref{equ-0408-07} is completed.

\medskip

\emph{Step 2. We prove \eqref{equ-513-1} with the aid of
\eqref{equ-0408-07} and a telescopic series method.}

It deserves mentioning that the similar method has already been
successfully applied to obtain observability inequalities
in \cite{AEWZ,BP, luis1, luis2, M10,PW,WWZZ}.

Arbitrarily fix $t\in (0, \tau_0)$ and $\varphi\in L^2(\R^n)$. We set
\begin{align}\label{equ-0408-14}
t_j=2^{-j}t, \quad j\in\N.
\end{align}
Applying \eqref{equ-0408-07}, where $\tau=2^{-1}t_j$ and $\varphi$ is replaced by $e^{-t_{j+1}\tilde{H}}\varphi$, we get
\begin{align}\label{equ-0408-15}
& g(2^{-1}t_j)\big\|e^{-t_j \tilde{H}}\varphi\big\|^2_{L^2(\R^n)}
-g(2^{-1}t_{j+1})\big\|e^{-t_{j+1} \tilde{H}}\varphi\big\|^2_{L^2(\R^n)}\nonumber\\
&\leq \int_{2^{-1}t_{j+1}}^{2^{-1}t_j}
{\big\|e^{-(s+t_{j+1}) \tilde{H}}\varphi\big\|^2_{L^2(E)}\d s}+\alpha_02^{-1}t_j\|
\varphi\|^2_{L^2(\R^n)} \nonumber\\
&\leq   \int_{t_{j+1}}^{t_j}
{\big\|e^{-s\tilde{H}}\varphi\big\|^2_{L^2(E)}\d s}+\alpha_0\tau_02^{-(j+1)}\|
\varphi\|^2_{L^2(\R^n)}.
\end{align}
Summing up  \eqref{equ-0408-15} from $j=0$ to $j=m\in \N^+$ leads to
\begin{align}\label{equ-0408-16}
&g({t}/{2})\big\|e^{-t\tilde{H}}\varphi\big\|^2_{L^2(\R^n)}
-g(2^{-1}t_{m+1})\big\|e^{-t_{m+1} \tilde{H}}\varphi\big\|^2_{L^2(\R^n)} \nonumber\\
&\leq\int_{t_{m+1}}^{t}{\big\|e^{-s\tilde{H}}f\big\|^2_{L^2(E)}\d s}+\alpha_0\tau_0\big(1-2^{-(m+1)}\big)\|\varphi\|^2_{L^2(\R^n)}.
\end{align}
Meanwhile, three facts are given in order. First, it follows by \eqref{equ-513-5}
that
$$
g(t)\rightarrow 0,\;\;\mbox{as}\;\;t\rightarrow 0;
$$
 Second,
 $$
 t_{m+1}=\frac{t}{2^{m+1}}\rightarrow 0,\;\;\mbox{as}\;\;m\rightarrow\infty;
 $$
 Third,
  \begin{eqnarray*}
  \big\|e^{-t_{m+1} \tilde{H}}\varphi\big\|_{L^2(\R^n)}\leq M\|\varphi\|^2_{L^2(\R^n)},
  \;\;\mbox{when}\; m\in \N.
  \end{eqnarray*}
   Combining the above  facts together, and  sending
   $m\rightarrow\infty$ in \eqref{equ-0408-16}, we find that  when
  $t\in (0,\tau_0)$ and $\varphi\in L^2(\mathbb{R}^n)$,
  \begin{align}\label{equ-0408-17}
g({t}/{2})\big\|e^{-t\tilde{H}}\varphi\big\|^2_{L^2(\R^n)}
\leq\int_{0}^{t}{\big\|e^{-s\tilde{H}}\varphi\big\|^2_{L^2(E)}\d s}
+\alpha_0\tau_0\|\varphi\|^2_{L^2(\R^n)}.
\end{align}

Finally, we will get  \eqref{equ-513-1}  from \eqref{equ-0408-17}. Indeed,
by \eqref{equ-513-2} and \eqref{equ-513-4}, we have $\frac{A}{N}<\tau_0$.
Thus by \eqref{equ-0408-17}
(with  $t=\frac{A}{N}$) and \eqref{equ-0512-002}, we see that when
 $\varphi\in L^2(\mathbb{R}^n)$,
\begin{eqnarray}\label{equ-0408-18}
\big\|e^{-\frac{A}{N} H}\varphi\big\|^2_{L^2(\R^n)}
&=&e^{\frac{2A\delta_0}{N}}\big\|e^{-\frac{A}{N} \tilde{H}}\varphi\big\|^2_{L^2(\R^n)}\nonumber\\
&\leq&\frac{e^{\frac{2A\delta_0}{N}}}{g(\frac{A}{2N})}\int_{0}^{\frac{A}{N}}
{\big\|e^{-sH}\varphi\big\|^2_{L^2(E)}\d s}+\beta \|\varphi\|^2_{L^2(\R^n)},
\end{eqnarray}
where
$$
\beta:=\frac{8NM^2\alpha_0\tau_0}{A}\exp\Big\{2c_1(2N)^{\frac{a}{b}}
+\frac{2A\delta_0}{N}\Big\}.
$$
Meanwhile, by   \eqref{equ-513-3}, \eqref{equ-513-4},
\eqref{equ-513-6} and the fact:
$$
e^{\frac{2A\delta_0}{N}}\leq e^{c_22^{-(b+1)}A},
$$
 we deduce that
 $$
 0<\beta<1.
 $$
This, along with   \eqref{equ-0408-18},
leads to \eqref{equ-513-1} where
 $$
 \alpha=\sqrt{\beta}\in(0,1),\; T=\frac{A}{N},\; C
 =\sqrt{\frac{e^{\frac{2A\delta_0}{N}}}{g(\frac{A}{2N})}}.
 $$

 Hence, we finish the proof of Lemma \ref{lem-0408-01}.
\end{proof}

 \begin{remark}\label{rmk-stable-general}
Lemma \ref{lem-0408-01} can be extended into what follows: (The proof is very similar to that of Lemma \ref{lem-0408-01}, we omit the details.)

\vskip 5pt
 {\it Let $X$ and $U$ be two real Hilbert spaces. Let  $H$ be a self-adjoint operator on $X$ so that $(-H)$ generates a $C_0$ semigroup $\{e^{-tH}\}_{t\geq 0}$ on $X$. Let $B$ be a linear bounded operator
 from $U$ to $X$. Let  $\{\pi_{k}\}_{k\in \N^+}$ be a family of orthogonal projections  on $X$. If
 there are positive constants $c_1$, $c_2$, $a$, $b$ and $M$ so that for each $k\in\N^+$,
  \begin{eqnarray}\label{2.3-6-10gw}
\|\pi_k \varphi\|_{X}\leq e^{c_1k^a}\|B^*\pi_k \varphi\|_{U},\; \mbox{when}\;\;\varphi\in X,
\end{eqnarray}
and
  \begin{eqnarray}\label{2.4-6-11gw}
\big\|(1-\pi_k) \big(e^{-tH}\varphi\big)\big\|_{X}
\leq M e^{-c_2tk^b}\|\varphi\|_{X},\;\mbox{when}\;\;\varphi\in X, \,t>0, \end{eqnarray}
then
 there is $\alpha\in (0, 1)$, $T>0$ and $C>0$ so that
\begin{eqnarray}\label{2.5-6-12gw}
\big\|e^{-TH}\varphi\big\|_{X}
\leq C\Big(\int_0^T\big\|B^*e^{-tH}\varphi\big\|^2_{U}\d t\Big)^{1/2}
+ \alpha \|\varphi\|_{X}\; \;\mbox{for all}\;\; \varphi\in  X.
\end{eqnarray}
}

According to  (\cite[Theorem 1]{TWX}), the inequality \eqref{2.5-6-12gw}
is equivalent to the stabilization of the control system:
\begin{align*}
(\partial_t+H)y(t)=Bu(t),\,\,\, t\geq 0,\;\;\mbox{with}\;\;u\in L^2(0, \infty; U).
\end{align*}

\end{remark}

\section{Two spectral inequalities}\label{sec-2.2}

Generally, in applications of  Lemma \ref{lem-0408-01},
it is easier to verify the dissipative inequality \eqref{equ-0408-06}
than the spectral  inequality \eqref{equ-0408-05}.
 In this section, we present spectral inequalities for the shifted
  fractional Laplacian and the shifted Hermite operator respectively.
   They
 not only play  important roles  in the proofs of
 Theorem \ref{thm-frac-laplace} and Theorem \ref{thm-hermite},
 but also may have independent significance. We  start with
 introducing some spectral projections.

\subsection{Spectral projections}\label{spectral projections}

 First, we consider the operator $H$ given by \eqref{equ-frac}.
 Define a family of orthogonal projections $\{\pi_k\}_{1\leq k\in\R}$ in
 the following manner:
  For each $k\in\R$ with  $k\geq 1$, let
  $\pi_k: \varphi(\in L^2(\mathbb{R}^n))\rightarrow\pi_k
  \varphi(\in L^2(\mathbb{R}^n))$
  be given by
 \begin{align}\label{equ-0520-2}
\pi_k \varphi= \left\{
\begin{array}{ll}
\mathcal{F}^{-1}\Big( \chi(|\xi|^s-c\leq k)\widehat {\varphi}(\xi)\Big),
 \quad \,\,\,\,\text{if}\,\,k+c>0,\\[0.3cm]
0,\,\,\,\quad\quad\quad\quad\quad\quad
\quad\quad\quad\quad\quad\quad \text{if}\,\,k+c\le 0,
\end{array}
\right.
\end{align}
where $\chi(|\xi|^s-c\leq k)$ denotes the characteristic function
of the set $\{\xi\in\R^n: |\xi|^s-c\leq k\}$.
 It deserves mentioning that the above  $\pi_k$  is exactly the usual
  spectral projection  $P_{(-\infty, k]}(H)$, associated with $H$
  given by \eqref{equ-frac}.

 We next consider the operator $H$ given by \eqref{equ-herm}.
 Recall several known facts on the Hermite operator
 $H_0:=-\Delta+|x|^2$:

\vskip 5pt
   \noindent{\it Fact one (see  \cite{Ti})}$\;$ We have
   $\sigma(H_0)=\{2k+n,\,\,\,k\in \mathbb{N} \}$. Thus by the spectral theorem (see, e.g. , \cite[p. 412]{BS}),  we see
\begin{align}\label{equ-418-1.3}
H_0=-\Delta+|x|^2=\sum_{k=0}^{\infty}{(2k+n) P_k},
\end{align}
where $P_k$ denotes the orthogonal projection onto the linear space spanned by the
eigenfunctions of $H_0$ associated with the eigenvalue $2k+n$.

\vskip 5pt
\noindent {\it Fact two (see  \cite{Ti})} $\;$ For each $k\in \mathbb{N} $, let
\begin{align}\label{equ-418-1.4}
\varphi_k(x)=\left(2^kk!\sqrt{\pi}\right)^{-\frac12}H_{k}(x) e^{-\frac{x^2}{2}},\;x\in \mathbb{R},
\end{align}
where  $H_{k}$ is the Hermite polynomial  given by
\begin{eqnarray}\label{5.5,11-20}
H_{k}(x)=(-1)^k e^{x^2}\frac{\d^k}{\d x^k}\big(e^{-x^2}\big), \; x\in \mathbb{R}.
\end{eqnarray}
 For each multi-index $\alpha=(\alpha_1, \alpha_2,\ldots, \alpha_n)$
 ($\alpha_i\in \mathbb{N}$), we define the following
  Hermite function by the tensor product:
\begin{align}\label{5.6,10-27}
\Phi_{\alpha}(x)=\prod_{i=1}^n {\varphi_{\alpha_i}(x_i)},
\;\;x=(x_1,\dots,x_n)\in \mathbb{R}^n.
\end{align}
Then for each $\alpha\in \mathbb{N}^n$ with $|\alpha|=k$,
$\Phi_{\alpha}$  is an eigenfunction of $H_0$ corresponding to the eigenvalue $2k+n$, and
$\{\Phi_{\alpha}\;:\;\alpha\in \mathbb{N}^n\}$ forms a complete orthonormal
basis in $L^2(\mathbb{R}^n)$.

\vskip 5pt

 \noindent {\it Fact three} $\;$  The authors in  \cite{BJP} built up
  spectral inequalities for finite combinations of Hermite functions when $E$
   is (i) an open subset; (ii)
 a weakly thick set (see \eqref{equ-0422-1} for the definition); (iii) a thick set.

 \vskip 5pt

 We now define   a family of orthogonal projections $\{\pi_k\}_{k\in \N^+}$
 (associated with $H$ given by \eqref{equ-herm}) in the following manner:
  For each $k\in\N^+$,
  let $\pi_k: \varphi(\in L^2(\mathbb{R}^n))\rightarrow\pi_k
  \varphi(\in L^2(\mathbb{R}^n))$ be given by
 \begin{align}\label{equ-0520-4}
\pi_k \varphi= \left\{
\begin{array}{ll}
\sum_{0\leq j\leq (k+c-n)/2} P_j\varphi,
\quad\quad\quad\quad \,\,\,\,\,\text{if}\,\,k+c\geq n,\\[0.3cm]
0,\,\,\,\quad\quad\quad\quad\quad\quad
\quad\quad\quad\quad\quad\quad \text{if}\,\,k+c< n,
\end{array}
\right.
\end{align}
where  $P_j$ is given by
\eqref{equ-418-1.3}. The above $\pi_k$ is exactly the
usual spectral projection $P_{(-\infty, k]}(H)$,
associated with $H$ given by \eqref{equ-herm}.

\subsection{Spectral inequalities}

\begin{lemma}\label{lemm-spec-frac}
Let $H$ be given by \eqref{equ-frac} and let
$\{\pi_k\}_{1\leq k\in\R}$ be defined by  \eqref{equ-0520-2}.
If $E$ is a thick set in $\R^n$, then there is $C>0$
so that
\begin{align}\label{equ-514-1}
\|\pi_k \varphi\|_{L^2(\R^n)}
\leq e^{Ck^{\frac{1}{s}}}\|\pi_k \varphi\|_{L^2(E)},
 \;\;\mbox{when}\;\; \varphi\in L^2(\R^n)\;\;\mbox{and}\;\;k\ge 1.
\end{align}
\end{lemma}
\begin{proof}
Let $s>0$ and $c\in \mathbb{R}$.
Arbitrarily fix  $k\in \mathbb{R}$ so that $k\geq 1$.
In the case that $k+c\leq 0$, \eqref{equ-514-1} is clearly true.
We now consider the case that $k+c>0$. By \eqref{equ-0520-2},
 we see that  the support of  $\widehat{\pi_k\varphi}$
 is contained in the ball $B(0,(k+c)^{1/s})$.
 Then by the Logvinenko-Sereda theorem (see \cite{HJ}, \cite[Theorem 1]{Kov}
 or  \cite[Lemma 2.1]{WWZZ}), we can find
  $C_1>0$ so that
\begin{align}\label{equ-0408-28}
\|\pi_k \varphi\|_{L^2(\R^n)}\leq
e^{C_1(1+(k+c)^{\frac{1}{s}})}\|\pi_k \varphi\|_{L^2(E)},
 \;\;\mbox{when}\;\;\varphi\in L^2(\R^n).
\end{align}
 Since $k\geq 1$, \eqref{equ-514-1} follows from  \eqref{equ-0408-28} at once.
 This completes the proof of Lemma \ref{lemm-spec-frac}.
 \end{proof}

\begin{lemma}\label{lemm-spec-herm}
Let $H$ be given by \eqref{equ-herm} and
let $\{\pi_k\}_{k\in\N^+}$ be given by \eqref{equ-0520-4}.
 If $E$ is a subset of positive measure in $\R^n$, then  there is $C>0$ so that
\begin{align}\label{equ-0520-5}
\|\pi_k \varphi\|_{L^2(\R^n)}\leq e^{\frac{n}{2}k\ln k+Ck}\|\pi_k \varphi\|_{L^2(E)},\;\;\mbox{when}\;\;\varphi\in L^2(\R^n)
\;\;\mbox{and}\;\;k\in \N^+.
\end{align}
\end{lemma}
\begin{proof}
We borrowed some idea from \cite{BJP}.  Arbitrarily fix
 $c\in \mathbb{R}$ and $k\in\N^+$. From  \eqref{equ-0520-4}, we see that
 it suffices to consider the non-trivial case
$$
k\ge k_0:=\max\{1, n-c\}.
$$
According to \eqref{equ-0520-4}, \eqref{equ-418-1.3},
\eqref{equ-418-1.4} and \eqref{5.6,10-27},  the range
 of $\pi_k$ is as
\begin{align}\label{equ-515-1}
\mathcal{E}_k= span\Big\{\Phi_{\alpha}:
 |\alpha|\le \frac{k+c-n}{2}, \;\alpha\in \mathbb{N}^n \Big\}.
\end{align}
To proceed, we recall the following two results from \cite{BJP}: First, it follows by
\cite[Lemma 4.2]{BJP} that  there is $c_n>0$ independent of $k$ so that
\begin{align}\label{equ-0419-1}
\|\varphi\|_{L^2(\R^n)}\leq \frac{2}{\sqrt{3}}
\|\varphi\|_{L^2(B(0, c_n\sqrt{k+c+1}))},
\;\;\mbox{when}\;\;\varphi\in \mathcal{E}_k.
\end{align}
(It deserves mentioning that though the above $\mathcal{E}_k$
 is slightly different from that in  \cite[Lemma 4.2]{BJP},
 the conclusion is  still true. This can be verified easily.)
Second, it follows from \cite[Lemma 4.4]{BJP} that
 if $\omega\subset \R^n$ satisfies $|\omega\bigcap B(0, R)|>0$,
 then each  polynomial  $P=P(x_1,\ldots,x_n)$, with
 $\deg P=d$,
  satisfies
\begin{align}\label{equ-0419-2}
\|P\|_{L^2(B(0, R))}\leq \frac{2^{2d+1}}{\sqrt{3}}
\sqrt{\frac{4|B(0, R)|}{|\omega\bigcap B(0, R)|}}F
\left(\frac{|\omega\bigcap B(0, R)|}
{|B(0, R)|}\right)\|P\|_{L^2(\omega\bigcap B(0, R))},
\end{align}
where
\begin{align}\label{equ-0419-3}
F(t):=\left(\frac{1+(1-\frac{t}{4})^{1/n}}
{1-(1-\frac{t}{4})^{1/n}}\right)^{d}, \;\;\; t\in (0,1].
\end{align}

We now prove \eqref{equ-0520-5}. Arbitrarily fix a
measurable subset $E\subset \mathbb{R}^n$, with  $|E|>0$.
Then there is $\varepsilon_0>0$ and $R_0>0$ so that
\begin{align}\label{equ-0419-4}
|E\bigcap B(0, R)|\ge\varepsilon_0,\;\;\mbox{when}\;\;R\ge R_0.
\end{align}
Meanwhile, it follows by \eqref{equ-418-1.4} that
for each  $\varphi\in\mathcal{E}_k$,
there is  a polynomial $P_k$, with $\deg P_k\leq k$, so that
\begin{align}\label{equ-0419-5}
\varphi(x)=P_k(x)e^{-\frac{|x|^2}{2}},\;\;x\in\R^n.
\end{align}
 Set
\begin{align}\label{const-k_1}
k_1:=\max\left\{k_0,\, ({R_0}/{c_n})^2-c\right\}.
\end{align}
There are only two possibilities  for the above fixed $k$: either $k> k_1$ or $k\leq k_1$.
We organize the rest of  the proof by two steps.

\vskip 5pt

\noindent{\it Step 1. We consider the case that $k> k_1$.}

  Given $\varphi\in\mathcal{E}_k$, we have
\begin{eqnarray}\label{equ-0419-6}
\|\varphi(x)\|_{L^2(\R^n)}&\leq& \frac{2}{\sqrt{3}}
\|P_k(x)\|_{L^2(B(0, c_n\sqrt{k+c+1}))}\nonumber\\
&\leq& C\varepsilon_0^{-\frac12}(k+c+1)^{\frac{n}{2}}F(t_k)
\|P_k\|_{L^2(E\bigcap B(0, c_n\sqrt{k+1}))}\nonumber\\
&\leq& C\varepsilon_0^{-\frac12}(k+c+1)^{\frac{n}{2}}
e^{\frac{c_n^2(k+c+1)}{2}}F(t_k)\|\varphi\|_{L^2(E)},
\end{eqnarray}
 where
\begin{align}\label{equ-0419-7}
t_k:=\frac{|E\bigcap B(0, c_n\sqrt{k+c+1})|}{|B(0, c_n\sqrt{k+c+1})|}\in(0,1].
\end{align}
In \eqref{equ-0419-6}, for the first inequality on Line 1,
  we used \eqref{equ-0419-1}, \eqref{equ-0419-5} and the
  inequality: $\|e^{-|\cdot|^2}\|_{L^{\infty}}\leq 1$
  (which is trivial);
for the second inequality on Line 2, we  used \eqref{equ-0419-2}
 and \eqref{equ-0419-4} (note that by \eqref{const-k_1}, one has $c_n\sqrt{k+c+1}> R_0$
when $k> k_1$); for the third  inequality on Line 3,
 we  used  \eqref{equ-0419-5} and the inequality:
$$
e^{\frac{|x|^2}{2}}\leq e^{\frac{c_n^2(k+c+1)}{2}},
\;\;\mbox{when}\;\;x\in B(0, c_n\sqrt{k+c+1}).
$$
We now estimate the upper bound of $F(t_k)$. By \eqref{equ-0419-3} where $d=k$,
 we have
 \begin{align}\label{equ-0419-8}
F(t_k)\leq 2^{k}\Big(1-\Big(1-\frac{t_k}{4}\Big)^{1/n}\Big)^{-k}.
\end{align}
To proceed, two facts are given in order. First,
  the following function is decreasing:
  $$
  t \mapsto \Big(1-\Big(1-\frac{t}{4}\Big)^{1/n}\Big)^{-k},\; t\in (0,1];
  $$
  Second, it follows by  \eqref{equ-0419-4} and \eqref{equ-0419-7} that
  $$
  c\varepsilon_0k^{-\frac{n}{2}}\le t_k\le 1,
  $$
   where $c=c(n)>0$ depends only on the dimension $n$.
Combining these facts together,  we obtain from \eqref{equ-0419-8} that
\begin{align}\label{equ-0419-8-1}
F(t_k)\leq 2^{k}\Big(1-\Big(1-\frac{c\varepsilon_0}{4}k^{-\frac{n}{2}}\Big)^{1/n}\Big)^{-k}\leq C^kk^{\frac{nk}{2}},
\end{align}
where $C=C(\varepsilon_0,n)>0$ is independent of $k$. Since
 $\pi_k\varphi\in \mathcal{E}_k$ for each $\varphi\in L^2(\R^n)$,
  it follows
  by \eqref{equ-0419-6} and \eqref{equ-0419-8-1} that
\begin{align}\label{equ-0408-23}
\|\pi_k \varphi\|_{L^2(\R^n)}\leq e^{\frac{n}{2}k\ln k+Ck}\|\pi_k \varphi\|_{L^2(E)},\;\;\mbox{when}\;\;\varphi\in L^2(\R^n),\;\; k>k_1,
\end{align}
where  $C>0$ is  independent of $k$.

\vskip 5pt

\noindent {\it Step 2. We consider the case that $k\leq k_1$.}

Since $|E|>0$, it follows by \eqref{equ-0419-2}
  that if $\varphi\in \mathcal{E}_{k}$ satisfies
  $\|\varphi\|_{L^2(E)}=0$,  then $\varphi= 0$ over $\mathbb{R}^n$.
This shows that $\|\cdot\|_{L^2(E)}$ is a norm on $\mathcal{E}_{k}$.
On the other hand, we see from  \eqref{equ-515-1} that
the subspace $\mathcal{E}_{k}$ is of finite dimension.
Hence,  the norm $\|\cdot\|_{L^2(E)}$ is equivalent to the norm
 $\|\cdot\|_{L^2(\R^n)}$. In particular, there is $C=C(k_1, E)>0$,
 independent of $k$, so that
\begin{eqnarray}\label{2.49-5-30wang}
\|\pi_k \varphi\|_{L^2(\R^n)}\leq C\|\pi_k \varphi\|_{L^2(E)},
\;\;\mbox{when}\;\;\varphi\in L^2(\R^n),\;\;k\leq k_1.
\end{eqnarray}

Finally, the spectral inequality \eqref{equ-0520-5} follows from \eqref{equ-0408-23}
and \eqref{2.49-5-30wang}. Thus, we finish the proof of Lemma \ref{lemm-spec-herm}.
\end{proof}

 \subsection{Comparison of  spectral inequalities}\label{subsec-3.3}
 This subsection concerns the difference between two spectral inequalities
 given by Lemma \ref{lemm-spec-frac} and Lemma \ref{lemm-spec-herm}
   respectively.

      We first consider the case that $H=H_0=-\Delta+|x|^2$,
      which is \eqref{equ-herm} with $c=0$.
      Let $\pi_k$, with $k\in\N^+$, be the projection
      (associated with  the above $H$)
     given by \eqref{equ-0520-4} where $c=0$. Let $E\subset \mathbb{R}^n$.
     With regard to the spectral inequality:
     \begin{align}\label{equ-0408-29}
     \|\pi_k \varphi\|_{L^2(\R^n)}\leq C(k,E)\|\pi_k
     \varphi\|_{L^2(E)}, \;\;\mbox{when}\;\; \varphi\in
     L^2(\R^n),\,\, k\geq 1,
     \end{align}
        the following interesting phenomena were revealed in
        the recent work \cite[Theorem 2.1]{BJP}:
        \begin{itemize}
          \item If $E$ is a non-empty open set, then one
          has \eqref{equ-0408-29} with  $C(k,E)=e^{Ck\ln k}$ for some $C=C(E,n)>0$;
          (It deserves mentioning that according to our Lemma \ref{lemm-spec-herm}, \eqref{equ-0408-29} with  $C(k,E)=e^{Ck\ln k}$ is also true when $E$ is any set of positive measure.)
          \item  If $E$ is a weakly thick set in the following sense:
        \begin{align}\label{equ-0422-1}
        \liminf_{R\rightarrow\infty}\frac{|E\bigcap B(0, R)|}{| B(0, R)|}>0,
        \end{align}
        then one has \eqref{equ-0408-29} with $C(k,E)=e^{Ck}$ for some constant $C=C(E,n)>0$;
          \item If $E$ is a thick set, then one has \eqref{equ-0408-29}
          with  $C(k,E)=e^{C\sqrt{k}}$ for some $C=C(E, n)>0$.
        \end{itemize}
                {\it From these, we see that  different geometry of $E$
                may lead to different growth order of  $C(k,E)$  in terms of $k$.}

                We next consider the case that  $H=(-\Delta)^{\frac{s}{2}}$
                with $s>0$, which is \eqref{equ-frac} with $c=0$.
                Let $\pi_k$, with $k\in\N^+$, be the projection (associated with $H$)
                defined by
                \eqref{equ-0520-2} where $c=0$.
                Let $E$ be a subset of $\mathbb{R}^n$.
                With regard to the spectral inequality
                 \eqref{equ-0408-29} in this case,
                we have what follows:
                 \begin{itemize}
                 \item According to the Logvinenko-Sereda theorem (see \cite [p.113] {HJ}),
                  the spectral inequality \eqref{equ-0408-29} holds
                 for some $C(k,E)$ if and only if $E$ is a thick set;
                 \item The above conclusion and Lemma \ref{lemm-spec-frac}
                  show further that the spectral inequality \eqref{equ-0408-29}
                 holds for $C(k,E)=e^{C(1+k^{\frac{1}{s}})}$
                 if and only if $E$ is a thick set.
                               \end{itemize}
                 {\it From these, we see that the geometry of $E$ can
                 influence the spectral inequality \eqref{equ-0408-29}
                 only in the manner: when $E$ is a thick set, \eqref{equ-0408-29} holds,
                  while when $E$ is not a thick set, \eqref{equ-0408-29} is not true.}

                 We now explain what causes the above-mentioned difference:
                 In the case when $H$ is given by \eqref{equ-herm} (with $c=0$),
                  $\mathcal{E}_k$ (the range of $\pi_k$ associated with $H$)
                  is of finite dimension, and is spanned by finite many Hermite functions.
                  Thus, for any subset $E$ of positive  measure, the norms
                  $\|\cdot\|_{L^2(\R^n)}$ and $\|\cdot\|_{L^2(E)}$ are equivalent.
                   This leads to the spectral inequality for any subset of positive measure.
                   On the other hand, in the case that $H=(-\Delta)^{\frac{s}{2}}$
                   with $s>0$, which is \eqref{equ-frac} with $c=0$,
                   the corresponding subspace  $\mathcal{E}_k$  is clearly not of finite dimension.
                    Moreover, it is translation invariant in the sense that
                    if $g(x)\in \mathcal{E}_k$, then $g(x-x_0)\in \mathcal{E}_k$
                     for all $x_0\in\R^n$ (since the support of $\widehat {g}(\cdot-x_0)$
                      is the same as that of $\widehat {g}(\cdot)$). Then by testing
                      the spectral inequality to the
                      function $g(\cdot-x_0)$, where $x_0$ is arbitrarily taken from $\R^n$,
                       one can prove that $E$ must satisfy
                       the  condition \eqref{equ-set-thick}.
                       These  will be exploited in detail later. (See
                   {\it Step 3} in the proof
                    of  Theorem \ref{thm-frac-laplace}.)

\section{Proof of main results}\label{sec-3}
  In this section, we will prove Theorem \ref{thm-frac-laplace}
  and Theorem \ref{thm-hermite} with the help of    Lemma \ref{lem-0408-01},
  Lemma \ref{lemm-spec-frac}
   and  Lemma \ref{lemm-spec-herm}.

\subsection{Proofs of Theorem \ref{thm-frac-laplace}}
First of all, we recall  that the operator $H$
given by  \eqref{equ-frac} is self-adjoint and
$(-H)$
generates an analytic semigroup. Now
we divide the proof into two steps.

\medskip

\noindent{\it Step 1. We show that $(i)\Rightarrow (ii)$.}

\medskip
 Suppose that   $E\subset \mathbb{R}^n$ is a thick set.
 To show $(ii)$, we  let  $H$ be given by \eqref{equ-frac}, with arbitrarily fixed $c\geq 0$ and $s>0$.
 Let $\pi_k$, with $k\in \N^+$, be the projection
 (associated with the aforementioned $H$) given by \eqref{equ-0520-2}.
  According to Lemma \ref{lemm-spec-frac}, the above
    $\{\pi_k\}_{k\in\N^+}$ and $E$ satisfy the spectral inequality \eqref{equ-0408-05},
 with  $a=1/s$ and  $c_1=C$ (where $C$ is given in \eqref{equ-514-1}).
 Meanwhile, by
  the Plancherel theorem, we see that for each $k\in\N^+$
  and each $\varphi\in L^2(\R^n)$,
\begin{eqnarray}\label{equ-0408-24-1}
\left\|(1-\pi_k) (e^{-tH}\varphi)\right\|_{L^2(\R^n)}
&=& \left\|\mathcal{F}^{-1}\left(\chi(|\xi|^s-c> k)e^{-t(|\xi|^{s}-c)}\widehat {\varphi}(\xi)\right)\right\|_{L^2(\R^n)}\nonumber\\
 &\leq& e^{-tk}\|\varphi\|_{L^2(\R^n)}, \quad t\geq 0,
\end{eqnarray}
 which leads to the dissipative inequality \eqref{equ-0408-06} with $M=b=c_2=1$.

Finally,  applying Lemma \ref{lem-0408-01} to the above $H$,
$\{\pi_k\}_{k\in\N^+}$ and $E$ leads to $(ii)$.

\medskip

\noindent{\it Step 2. We show that $(ii)$ $\Rightarrow (i)$.}

\medskip

Let $E\subset \mathbb{R}^n$ be measurable. According to  Theorem \ref{thm-txw} in Section \ref{sec-abstract},
it suffices to show what follows:

\noindent {\it Statement A. If there is $C>0$, $T>0$ and $\alpha\in (0,1)$ so that for each
$\varphi \in L^2(\R^n)$,
\begin{align}\label{equ-514-5}
 \left\|e^{-TH}\varphi\right\|_{L^2(\R^n)} \leq C \Big(\int_0^T
 \int_E\left|(e^{-tH}\varphi)(x)\right|^2\d x \d t\Big)^{1/2} + \alpha \|\varphi\|_{L^2(\R^n)},
\end{align}
where $H=(-\Delta)^{\frac{s}{2}}-c$ with some $s>0$ and $c\geq 0$,
 then $E$ is a thick set.}

Before proving {\it Statement A},  we mention that a similar statement (i.e., \eqref{equ-514-5}, with $\alpha=c=0,s=2$) is proved  in \cite[Theorem 1.1]{WWZZ}. The new ingredient here is that we shall add a parameter to eliminate the impact  of the term $\alpha \|\varphi\|_{L^2(\R^n)}$ in \eqref{equ-514-5}.

We now show {\it Statement A}. Suppose that there is $s>0$ and $c\geq 0$  so that  \eqref{equ-514-5}
holds for $H=(-\Delta)^{\frac{s}{2}}-c$.
 Given  $x_0\in \R^n$, we define  a function
\begin{align}\label{equ-413-1}
u(t,x;l):= e^{ct}(t+l)^{-\frac{n}{s}}g\bigg(\frac{x-x_0}
{(t+l)^{\frac{1}{s}}}\bigg), \quad t\geq 0, x\in \R^n,
\end{align}
where $l>0$ is a parameter (which will be determined later) and
 $g$  is the inverse Fourier transform of  $e^{-|\cdot|^{s}}$, i.e.,
\begin{align}\label{equ-A-2}
g(x)=\left(\mathcal{F}^{-1}e^{-|\cdot|^{s}}\right)(x)
=\frac{1}{(2\pi)^n}\int_{\R^n}{e^{ix\cdot\xi}e^{-|\xi|^{s}}\,
\mathrm d\xi},\;\; x\in \mathbb{R}^n.
\end{align}
Three facts on the above $g$ are given in order. First,
 it is clear that $g$ is a smooth function;
Second,  \cite[Theorem 2.1]{BG} gives the following point-wise estimate:
 for some $C_1>0$,
\begin{align}\label{equ-A-3}
|g(x)|\leq \frac{C_1}{(1+|x|^2)^{\frac{n+s}{2}}},\;\;x\in\R^n;
\end{align}
Third,  it follows from  \eqref{equ-A-2} that each $u(t,x;l)$
 (with $l>0$) satisfies the following  fractional heat equation:
$$
(\partial_{t}+(-\Delta)^{\frac{s}{2}})u(t,x;l) =cu(t,x;l),
 \; t>0, x\in \mathbb{R}^n; \;\;\; u(0,x;l) = \varphi(x;l),\; x\in \mathbb{R}^n,
$$
with
$$
 \varphi(x;l):= l^{-\frac{n}{s}}g\left(\frac{x-x_0}
 {l^{\frac{1}{s}}}\right),\;\;x\in \mathbb{R}^n.
$$
(It is clear that $\varphi(\cdot;l)\in L^2(\mathbb{R}^n)$.)

By a direct computation, we have that for some absolute constant $C_2>0$,
\begin{align}\label{equ-413-3}
\|u(t,\cdot;l)\|_{L^2(\R^n)} = C_2(t+l)^{-\frac{n}{2s}}e^{ct}, \;\;\mbox{when}\;\; t\geq 0,\; l>0.
\end{align}
Using  \eqref{equ-514-5} (where  $\varphi(\cdot)=\varphi(\cdot;l)$),
 \eqref{equ-413-3} and the identity:
\begin{align}\label{equ-0520-1}
\left(e^{-tH}\right)\varphi(x;l)=e^{ct}e^{-t(-\Delta)^{\frac{s}{2}}}\varphi(x;l),
\;\;l>0,\; t\ge0,\; x\in \mathbb{R}^n,
\end{align}
we deduce that when $l>0$
\begin{align}\label{equ-413-4}
 C_2(T+l)^{-\frac{n}{2s}}\leq Ce^{cT}\Big(\int_{0}^T\int_{E}\left|
 e^{-t(-\Delta)^{\frac{s}{2}}}\varphi(x;l)\right|^2\d x\d t\Big)^{1/2}
 + C_2\alpha l^{-\frac{n}{2s}}.
\end{align}
Let
\begin{align}\label{equ-413-5}
l_0 := \frac{T}{\left(\frac{2}{1+\alpha}\right)^{\frac{2s}{n}}-1}.
\end{align}
Then  a direct calculation leads to
$$
(T+l_0)^{-\frac{n}{2s}} -\alpha l_0^{-\frac{n}{2s}}
= \frac{1-\alpha}{2} l_0^{-\frac{n}{2s}}>0.
$$
This, along with \eqref{equ-413-4} (where $l=l_0$), yields
\begin{align}\label{equ-413-6}
C_3^2\leq  C^2e^{2cT}\int_{0}^T\int_{E}\left|e^{-t(-\Delta)^{\frac{s}{2}}}
\varphi(x; l_0)\right|^2\d x\d t,
\end{align}
where
$$
C_3:=\frac{C_2(1-\alpha)l_0^{-\frac{n}{2s}}}{2}>0.
$$
  Related to \eqref{equ-413-6}, we have the following observations: First, for each $L>0$,
\begin{multline}\label{equ-514-7}
 \int_{0}^T\int_{E}\left|e^{-t(-\Delta)^{\frac{s}{2}}}\varphi(x; l_0)\right|^2\d x\d t\\
  =  \int_{0}^T\int_{E\bigcap B^c(x_0,L)}\big|
  e^{-t(-\Delta)^{\frac{s}{2}}}\varphi(x; l_0)\big|^2\d x\d t
  +  \int_{0}^T\int_{E\bigcap B(x_0,L)}
  \big|e^{-t(-\Delta)^{\frac{s}{2}}}\varphi(x; l_0)\big|^2\d x\d t.
\end{multline}
Second, when
 $t\in [0,T]$ and $L>0$,
\begin{eqnarray}\label{equ-418-1}
\int_{B^c(x_0,L)}{\left|e^{-t(-\Delta)^{\frac{s}{2}}}\varphi(x; l_0)\right|^2\, \mathrm dx}&=&\int_{B^c(x_0,L)}{\left|e^{-ct}u(t,x; l_0)\right|^2\, \mathrm dx}\nonumber\\
&\leq& C_1^2(t+l_0)^{-\frac{2n}{s}}\int_{B^c(x_0,L)}
{\bigg(1+\frac{|x-x_0|^2}{(t+l_0)^{2/s}}\bigg)^{-n-s}\, \mathrm dx}\nonumber\\
&\leq & C_1^2(t+l_0)^{-\frac{2n}{s}}\int_{|y|> L/(t+l_0)^{1/s}}
{(1+|y|^2)^{-n-s}\, \mathrm dy}\nonumber\\
&\leq & C_1^2l_0^{-\frac{2n}{s}}\int_{|y|> L/(T+l_0)^{1/s}}{(1+|y|^2)^{-n-s}\, \mathrm dy}.
\end{eqnarray}
(In \eqref{equ-418-1},  for the  equality on Line 1, we used \eqref{equ-0520-1};
for the inequality on Line 2, we used \eqref{equ-413-1} and  \eqref{equ-A-3};
for the inequality on Line 3, we used the change of variable $x-x_0=(t+l_0)^{1/s}y$;
for the inequality on Line 4,
 we used the fact that  $0\le t\leq T$.) Third, by \eqref{equ-418-1}
 and the fact that  $(1+|\cdot|^2)^{-n-s}\in L^1(\R^n)$,
 we can choose $L_0=L_0(s, T, n, \alpha)>0$   so that
\begin{align}\label{equ-413-7}
C^2e^{2cT}\int_{0}^T\int_{E\bigcap B^c(x_0,L_0)}\left|
e^{-t(-\Delta)^{\frac{s}{2}}}\varphi(x;l_0)\right|^2\d x\d t  \leq \frac{C^2_3}{2}.
\end{align}
Now, combining \eqref{equ-413-6}, \eqref{equ-514-7} and \eqref{equ-413-7} together, we see
\begin{align}\label{equ-413-8}
 \frac{C_3^2}{2}\leq  C^2e^{2cT}\int_{0}^T\int_{E\bigcap B(x_0,L_0)}\left|e^{-t(-\Delta)^{\frac{s}{2}}}\varphi(x;l_0)\right|^2\d x\d t.
\end{align}
Meanwhile,  it follows from \eqref{equ-413-1}, \eqref{equ-A-3} and \eqref{equ-0520-1}   that when $x\in B(x_0,L_0)$ and $t\in(0, T)$,
\begin{align}\label{3.18-5-31w}
   \left|e^{-t(-\Delta)^{\frac{s}{2}}}\varphi(x;l_0)\right|
   =e^{-ct}|u(t,x;l_0)| \leq C_1l_0^{-\frac{n}{s}}.
\end{align}
From  \eqref{equ-413-8} and \eqref{3.18-5-31w}, we find
$$
 \frac{C_3^2}{2}\leq  \left|E\bigcap B(x_0,L_0)\right| T C^2C^2_1e^{2cT}l_0^{-\frac{2n}{s}}.
$$
This shows that for some $C_4>0$  independent of  $x_0\in \R^n$,
\begin{align}\label{4.18-6-8wang}
\left|E\bigcap B(x_0,L_0)\right|\ge C_4.
\end{align}
Because
$$
B(x_0,L_0)\subset Q_{L_1}(x_0)\;\;\mbox{with}\;\;L_1:=2L_0,
$$
we see from \eqref{4.18-6-8wang} that for some  $\gamma>0$,
$$
\left|E\bigcap Q_{L_1}(x_0)\right|\ge \gamma L_1^n.
$$
Since $x_0$ can be  arbitrarily taken from $\R^n$, the above shows that $E$ is a thick set in $\R^n$.

Hence, we finish the proof of Theorem \ref{thm-frac-laplace}.\qed

\subsection{Proof of Theorem  \ref{thm-hermite}}

\medskip
First of all,
we recall that  the operator $H$ given by \eqref{equ-herm}
is self-adjoint and $(-H)$ generates an analytic semigroup.
We organize the proof  by the following two steps:

\medskip

\noindent{\it Step 1. We show that $(i)\Rightarrow (ii)$.}

Suppose that $E\subset \mathbb{R}^n$ is a subset of positive measure.
Let $H$ be given by \eqref{equ-herm}, with an arbitrarily fixed $c\geq n$.
Let $\pi_k$, with $k\in\N^+$, be the projection (associated with the aforementioned $H$)
given by \eqref{equ-0520-4}.
 According to  Lemma \ref{lemm-spec-herm},  $\{\pi_k\}_{k\in\N^+}$ and
 $E$ satisfy the spectral inequality \eqref{equ-0408-05} for some  $a>1$ and  $c_1>0$.
Meanwhile, by
 the spectral theorem (see, e.g. , \cite[p. 412]{BS}) and \eqref{equ-0520-4}, it follows that
 when $\varphi\in L^2(\mathbb{R}^n)$ and $k\in\N^+$,
\begin{eqnarray}\label{equ-0419-0}
\big\|(1-\pi_k) (e^{-tH}\varphi)\big\|_{L^2(\R^n)}
&=& \big\|\sum_{j> (k+c-n)/2}^{\infty}e^{-(2j+n-c)t}P_j \varphi\big\|_{L^2(\R^n)}\nonumber\\
 &\leq& e^{-kt}\|\varphi\|_{L^2(\R^n)},\;\;t>0.
\end{eqnarray}
 From \eqref{equ-0419-0}, we see that $H$ and $\{\pi_k\}_{k\in\N^+}$  satisfy
 the dissipative inequality \eqref{equ-0408-06} with  $M=b=c_2=1$.

 Now, by applying
  Lemma \ref{lem-0408-01} to the above $H$, $\{\pi_k\}_{k\in\N^+}$ and $E$,
  we get $(ii)$.

\medskip

\noindent{\it Step 2. We show that $(ii)\Rightarrow (i)$.}

Suppose that $E$ is a {\it  stabilizable set} for   \eqref{equ-1} with \eqref{equ-herm} where $c\geq n$
 is arbitrarily fixed. Then by  Theorem \ref{thm-txw} in Section \ref{sec-abstract}, there is $\alpha\in(0,1)$, $T>0$ and  $C> 0$ so that
\begin{equation}\label{equ-0423-34}
\big\|e^{-TH}\varphi\big\|_{L^2(\R^n)}\leq C\Big(\int_{0}^T\int_{E}|e^{-tH}\varphi|^2\d x\d t\Big)^{1/2}+\alpha\|\varphi\|_{L^2(\R^n)}\;\;\mbox{for all}\;\;\varphi\in L^2(\R^n),
\end{equation}
where $H$ is given by \eqref{equ-herm}.
Meanwhile, it follows from  \eqref{5.6,10-27} that $\Phi_{0}$, defined by
$$
\Phi_{0}(x)=\pi^{n/4}e^{-\frac{|x|^2}{2}},\;\;x\in \mathbb{R}^n,
$$
is the $L^2$-normalized eigenfunction of the Hermite operator $H_0=-\Delta+|x|^2$ corresponding to the eigenvalue $\lambda=n$. Thus, we have
$$
\big(e^{-tH}\Phi_{0}\big)(x)=e^{-(n-c)t}\Phi_{0}(x), \quad \mbox{when}\,\, t\ge 0,\,  x\in \R^n.
$$
This, along with \eqref{equ-0423-34} where $\varphi=\Phi_{0}$, yields
\begin{equation}\label{equ-0423-35}
0<1-\alpha\leq e^{(c-n)T}-\alpha\leq C\pi^{n/4}T^{\frac12}e^{(c-n)T}\Big(\int_{E}{e^{-|x|^2}\d x}\Big)^{\frac12},
\end{equation}
which leads to  $|E|>0$.

Hence, we complete the proof of Theorem \ref{thm-hermite}.\qed

\begin{remark}\label{rmk-0430-1}
With the aid of the proof of Theorem \ref{thm-frac-laplace},
we will see that the condition $0<a<b$ in \eqref{equ-0408-06-1-1} cannot be dropped in general.
Here are two counterexamples:

\begin{itemize}
\item Let $H$ be given by \eqref{equ-frac} where $c=0$ and $0<s<1$. Let
$\pi_k$, with $k\in\N^+$, be the projection (associated with $H$) given by \eqref{equ-0520-2} (where $c=0$).
Let
 $$
 E:=\{x\in \R^n:\,\,|x|\ge 1\}.
 $$
 Clearly, it is   a thick set in $\R^n$.
 Moreover,  from {\it Step 1}  in the proof of Theorem \ref{thm-frac-laplace}, we see that in this case,
 \eqref{equ-0408-05} (with the above $E$ and  $a=\frac{1}{s}$) and   \eqref{equ-0408-06} (where $b=1$) are true. Thus we have $a>b>0$. However it was proved in \cite[Theorem 3]{Ko} that the above $E$ is not an {\it observable set} for \eqref{equ-1} with \eqref{equ-frac} where $0<s<1$ and $c=0$.

\item  Let $H$ be given by \eqref{equ-herm} where $c=0$. Let   $\pi_k$, with $k\in\N^+$, be
 the projection (associated with $H$)
 given by \eqref{equ-0520-4} (where $c=0$).  Let
 $$
 E:=\{x=(x_1,\dots,x_n)\in\R^n\;:\; x_1>0\}.
 $$
Then it follows from \cite[Theorem 2.1]{BJP} that  \eqref{equ-0408-05} (with the above $E$ and $a=1$)  and   \eqref{equ-0408-06} (with $b=1$) are true. Thus  we have  $a=b=1$ in this case. However it was proved in \cite[Theorem 1.10]{M09} that any half space is not an {\it observable set} for  \eqref{equ-1} with \eqref{equ-herm}.
\end{itemize}

Besides, we would like to mention what follows: In the critical case $a=b=1$, the observability inequality can still hold if one has  a further logarithmic assumption  in the spectral inequality \eqref{equ-0408-05}.
 We refer to \cite[Theorem 5]{DM} for this refined version of Lebeau-Robbiano strategy.
     \end{remark}

\section{ Proofs of Theorem \ref{thm-general-1} and  Theorem \ref{thm-general-2}}\label{sec-4}

In this section, we will  use different approaches to prove
Theorem \ref{thm-general-1} and  Theorem \ref{thm-general-2}.
These approaches provide explicit expressions for the feedback operator $K$  in \eqref{equ-stable-1}.

\subsection{Proof of Theorem \ref{thm-general-1}}

By  ($\textbf{b}_1$)  in Subsection \ref{concept},  Theorem \ref{thm-general-1} is an immediate consequence of the next Lemma \ref{thm-general}.

\begin{lemma}\label{thm-general}
 Assume that  $E$ is a thick set in $\R^n$.
 Suppose that $H=-\Delta+V(x)$ where $V$ satisfies {\it Condition I}.
   Then there is $\omega>0$ so that
$$
\big\|e^{-t(H+\chi_E)}\big\|_{\mathcal{L}(L^2(\R^n))}\leq  e^{-\omega t}
\;\;\mbox{for each}\;\; t\geq 0.
$$
\end{lemma}
\begin{proof}
Notice that  $H+\chi_E$ is closed and its domain is dense in $L^2(\R^n)$. Thus, according to
the Hille-Yosida theorem, it  suffices to show what follows:  There exists $\omega>0$ so that
\begin{align}\label{equ-425-1}
\|(H+\chi_E+\lambda)^{-1}\|_{\mathcal{L}(L^2(\R^n))}\leq \frac{1}{\lambda+\omega}\;\;\mbox{for each}\;\; \lambda\in  (-\omega, \infty).
\end{align}
To this end, we first claim that there exists $\omega>0$ so that
\begin{align}\label{equ-425-2}
  ((H+\chi_E)\varphi, \varphi) \geq \omega \int_{\R^n}|\varphi|^2\;\;\mbox{for each}\;\; \varphi\in C_0^\infty(\R^n).
\end{align}
 For this purpose,  three observations are given in order. First, since $V$ holds  {\it Condition I},
  there is  $\delta\in (0,1)$ so that
\begin{align}\label{equ-425-3}
  ((H+\chi_E)\varphi, \varphi) \geq (1-\delta)\int_{\R^n}|\nabla \varphi|^2\d x + \int_E |\varphi|^2\d x\;\;\mbox{for each}\;\; \varphi\in C_0^\infty(\R^n).
\end{align}
Second, given $N\geq 1$, let
  $\pi_{N}$  be the projection given by
  \eqref{equ-0520-2} where $s=1$, $c=0$, $k=N$. Then we have
\begin{align}\label{5.4-6-11wang}
\widehat{\pi_{N}\varphi}(\xi) = \chi_{\{|\xi|\leq N\}} \widehat{\varphi}(\xi),\;\;\; \xi\in \mathbb{R}^n, \;\;\varphi\in L^2(\mathbb{R}^n).
\end{align}
  Since $E$ is a thick set, by Lemma \ref{lemm-spec-frac} (where $H=(-\Delta)^{1/2}$
and $\pi_k$ is given by \eqref{equ-0520-2} with $s=1$, $c=0$, $k=N$), we
 can find   some  $C=C(E,n)>0$ so that for each $N\geq 1$,
\begin{align}\label{equ-425-4}
\|\pi_{N}\varphi\|_{L^2(\R^n)}\leq  e^{CN}\|\pi_{N}\varphi\|_{L^2(E)}\;\;\mbox{for each}\;\;  \varphi\in L^2(\R^n).
\end{align}
Third, let $N\ge 1$ satisfy
$$
\omega := \min \Big\{ (1-\delta)N^2-2, \; \frac{1}{2}e^{-2CN}\Big\}>0.
$$
Then we have that when $\varphi\in C_0^\infty(\R^n)$,
\begin{align} \label{equ-425-5}
\mbox{RHS}\; \eqref{equ-425-3} &\geq (1-\delta) \int_{\R^n}|(1-\pi_{N})\nabla \varphi|^2\d x  + \frac{1}{2}\int_{E}|\pi_{N}\varphi|^2\d x - 2\int_{E}|(1-\pi_{N})\varphi|^2\d x\nonumber\\
&\geq \Big((1-\delta)N^2-2\Big)\int_{\R^n}|(1-\pi_{N}) \varphi|^2\d x + \frac{1}{2}e^{-2CN}\int_{\R^n}|\pi_{N}\varphi|^2\d x\nonumber\\
&\geq \omega \int_{\R^n}|\varphi|^2\d x.
\end{align}
(Here, \mbox{RHS}\; \eqref{equ-425-3} denotes the right hand side of \eqref{equ-425-3}.)
In  \eqref{equ-425-5}, for the inequality on Line 1,  we used the inequality:
$$
\int_{\R^n}|\nabla\varphi|^2\d x=\int_{\R^n}|(1-\pi_{N}) \nabla\varphi|^2\d x+\int_{\R^n}|\pi_{N}\nabla \varphi|^2\d x \geq \int_{\R^n}|(1-\pi_{N}) \nabla\varphi|^2\d x
$$
and the inequality\footnote{It follows from  the fact that  $\varphi = \pi_{N}\varphi+(1-\pi_{N})\varphi$
and the elementary inequality:
$|a+b|^2\geq \frac{1}{2}|a|^2-2|b|^2$, with $a=\pi_{N}\varphi$ and $b=(1-\pi_{N})\varphi$.}:
\begin{eqnarray*}
\int_{E}|\varphi|^2\d x \geq \frac{1}{2}\int_{E}|\pi_{N}\varphi|^2\d x - 2\int_{E}|(1-\pi_{N})\varphi|^2\d x;
\end{eqnarray*}
 for  the inequality on Line 2,
 we used the fact
 $$
\int_{\R^n}|(1-\pi_{N}) \nabla\varphi|^2\d x=\int_{\R^n}\chi_{\{|\xi|\geq N\}}|\xi|^2|\hat{\varphi}|^2\d \xi \geq N^2\int_{\R^n}|(1-\pi_{N}) \varphi|^2\d x
$$
  and the spectral inequality \eqref{equ-425-4}.
 Now by  \eqref{equ-425-3} and \eqref{equ-425-5}, we get \eqref{equ-425-2}.

  Finally, by \eqref{equ-425-2}, we have that when  $\lambda >-\omega$ and $\varphi\in C_0^\infty(\R^n)$,
$$
0 \leq (\lambda+\omega) \int_{\R^n}|\varphi|^2 \leq ((H+\chi_E+\lambda)\varphi, \varphi) \leq \|(H+\chi_E+\lambda)\varphi\|_{L^2(\R^n)}\|\varphi\|_{L^2(\R^n)}.
$$
  This shows that when $\varphi\in C_0^\infty(\R^n)$,
$$
(\lambda+\omega)\|\varphi\|_{L^2(\R^n)}\leq \|(H+\chi_E+\lambda)\varphi\|_{L^2(\R^n)}.
$$
Using a density argument to the above leads to \eqref{equ-425-1}.
Thus we finish the proof of Lemma \ref{thm-general}.
\end{proof}

\subsection{Proof of Theorem \ref{thm-general-2}}

Suppose that $V$ satisfies  {\it Condition II}.
  Then
  the operator $H=-\Delta +V$ is self-adjoint and has compact resolvent in $L^2(\R^n)$ (see \cite{BS})).
    Thus,
 $$
 \sigma(H)=\{\lambda_j\}_{j\in\N^+}\;\mbox{with}\; \lambda_1\leq \lambda_2\cdots\leq\lambda_k\leq\cdots\rightarrow\infty.
 $$
 Let $\varphi_j$, with $j\in\N^+$, be the $L^2$-normalized eigenfunction of $H$ with respect to $\lambda_j$,
 i.e.,
 \begin{align}\label{equ-426-1}
H\varphi_j = (-\Delta+V)\varphi_j = \lambda_j \varphi_j\;\;\mbox{and}\; \|\varphi_j\|_{L^2(\mathbb{R}^n)}=1, \quad j=1,2,\cdots.
\end{align}
Then  $\{\varphi_j\}_{j\in\N^+}$ forms a complete orthonormal basis  in $L^2(\R^n)$.

If $\lambda_1>0$, then  the conclusion in Theorem \ref{thm-general-2}   is clearly true. Thus,
without loss of generality,  we can assume    $\lambda_1\leq 0$ from now on. Hence,
            there is $N\in\N^+$ so that
\begin{align}\label{equ-426-2}
 \lambda_1\leq \lambda_2\leq \cdots \leq \lambda_N\leq 0 < \lambda_{N+1} \leq \lambda_{N+2} \leq \cdots.
\end{align}

Before continuing with the proof of Theorem \ref{thm-general-2}, we give several lemmas
concerning with the eigenfunctions of $H$.

\begin{lemma}\label{lem-eigen-1}
Assume that $V$ satisfies  {\it Condition II}. Then each eigenfunction  of $H=-\Delta+V$ is
 continuous on $\R^n$.
\end{lemma}
\begin{proof}
This is a direct consequence of \cite[Theorem C.2.4, p.497]{Simon}.
\end{proof}

\begin{lemma} \label{ucp}
Suppose that  $V\in L^\infty_{loc}(\R^n)$.  Let $u\in H^2_{loc}(\R^n)$ be a solution to the equation:
$$
\Delta u = V(x)u  \quad \mbox{ in }\;\; \R^n.
$$
If $u$ vanishes on a nonempty open set, then $u$ is identically zero.
\end{lemma}
\begin{proof}
See, e.g., \cite{SS80}.
\end{proof}

To state the next lemma, we introduce the following notation:
Given  a nonempty open set $E\subset\R^n$,
we define the matrix
\begin{align}\label{equ-51-4}
\mathbb{A} := \mathbb{A}_E= \left(
\begin{array}{cccc}
(\varphi_1,\varphi_1)_E & (\varphi_1,\varphi_2)_E & \cdots & (\varphi_1,\varphi_N)_E\\
(\varphi_2,\varphi_1)_E & (\varphi_2,\varphi_2)_E & \cdots & (\varphi_2,\varphi_N)_E\\
\vdots & \vdots & \cdots & \vdots\\
(\varphi_N,\varphi_1)_E & (\varphi_N,\varphi_2)_E & \cdots & (\varphi_N,\varphi_N)_E
\end{array}
\right),
\end{align}
where \begin{align}\label{equ-51-5}
(\varphi_i,\varphi_j)_E := \int_{E}\varphi_i(x) \varphi_j (x)\d x, \quad 1\leq i,j\leq N.
\end{align}
(Recall  $\varphi_j$, with $j\in\N^+$, is given by \eqref{equ-426-1}.)

\begin{lemma}\label{lem-invert}
Suppose that $V$ satisfies  {\it Condition II}.
Let $\mathbb{A}$ be given by \eqref{equ-51-4}, associated with a  nonempty open subset $E\subset\R^n$. Then
\begin{eqnarray}\label{5.10-6-1wang}
\det \mathbb{A} \neq 0.
\end{eqnarray}
\end{lemma}
\begin{remark} $(i)$  In a slightly different setting,
Lemma \ref{lem-invert} was given in \cite[p.42]{B}, without proof.
 For the  reader's convenience, we provide a detailed proof here.
 $(ii)$ We mention that a  quantitative lower bound for the
 norm of the matrix $\mathbb{A}$
 (where $V=0$)  is obtained in \cite[Lemma 3.1]{LW} with the aid of
 the
  Lebeau-Robbiano spectral inequality. However, for the current case,  we don't
        have the corresponding spectral inequality.
         (At least, we do not find it in any published paper.)
         \end{remark}

\begin{proof}[Proof of Lemma \ref{lem-invert}.]
By  contradiction, we suppose  that  \eqref{5.10-6-1wang}  is not true. Then
 there exists $\textbf{0}\neq z=(z_1,z_2,\cdots,z_N)^T \in \R^N$
 so that $\mathbb{A} z =\textbf{0}$, which gives
\begin{align}\label{equ-51-6}
z^T\mathbb{A} z =0.
\end{align}
By \eqref{equ-51-4} and \eqref{equ-51-5}, we can rewrite \eqref{equ-51-6} as
\begin{align}\label{equ-51-8}
 \int_E \Big| \sum_{1\leq j\leq N}z_j \varphi_j(x)  \Big|^2 \d x=0.
\end{align}
  Meanwhile, according to Lemma \ref{lem-eigen-1}, the function
   $ \sum_{1\leq j\leq N}z_j \varphi_j(\cdot)$ is continuous on $\R^n$. This, together with \eqref{equ-51-8}, leads to
\begin{align}\label{equ-51-9}
\sum_{1\leq j\leq N}z_j \varphi_j(x) =0, \quad  \mbox{when}\;\; x\in E.
\end{align}
With respect to $\lambda_1,\dots, \lambda_N$, there are only two possibilities:
 either they are the same or  at least two of them are different.

  In the first case that  they are the same, we have
   $\lambda_1=\lambda_2=\cdots =\lambda_N=\lambda$ for some $\lambda\in \R$.
   Then
   $$
   (-\Delta +V)\varphi_j = \lambda \varphi_j\;\;\mbox{over}\;\;\mathbb{R}^n\;\;
   \mbox{for all}\;\;j=1,\dots,N,
   $$
      consequently,
\begin{eqnarray}\label{5.14-6-1wang}
(-\Delta +V)\Big(\sum_{1\leq j\leq N}z_j \varphi_j\Big)=\lambda \Big(\sum_{1\leq j\leq N}z_j \varphi_j\Big) \quad \mbox{over}\;\; \R^n.
\end{eqnarray}
Since $\sum_{1\leq j\leq N}z_j \varphi_j\in H_{loc}^2(\R^n)$ and because
$\sum_{1\leq j\leq N}z_j \varphi_j=0$ over $E$ (see \eqref{equ-51-9}),
by   \eqref{5.14-6-1wang}, we can apply   Lemma \ref{ucp} to conclude that
\begin{align}\label{equ-51-12}
\sum_{1\leq j\leq N}z_j \varphi_j=0  \quad  \mbox{over}\;\; \R^n.
\end{align}
Meanwhile, since  $\varphi_1,\dots,\varphi_N$  are linearly independent,
it follows from \eqref{equ-51-12}  that $z_1 =z_2=\cdots=z_N=0$,
 which contradicts $z\neq \textbf{0}$.

 In the second case that at least two among $\lambda_1,\dots, \lambda_N$ are different, we can let
      $\mu_1,\cdots,\mu_\ell$ (with $\ell\in [2,N]$ and $\mu_1<\cdots<\mu_l$) be all distinct values among
       $\{\lambda_1,\cdots,\lambda_N\}$.
       Then by \eqref{equ-426-2}, we can decompose  $\{\lambda_1,\cdots,\lambda_N\}$ as $\ell$ groups:
$$
\begin{array}{cccc}
\underbrace{\lambda_{j_0+1}=\cdots=\lambda_{j_1}}&
<\underbrace{\lambda_{j_1+1}=\cdots=\lambda_{j_2}}&<\cdots &< \underbrace{\lambda_{j_{\ell-1}+1}=\cdots=\lambda_{j_\ell}}\\
=\mu_1 & =\mu_2 & \cdots & =\mu_\ell
\end{array}
$$
where $j_0=0$ and $j_\ell=N$. Thus, we can rewrite  \eqref{equ-51-9} as:
\begin{align}\label{equ-54-1}
\sum_{k=1}^\ell \phi_k =0 \quad  \mbox{over} \;\; E,
\end{align}
with
\begin{align}\label{equ-54-2}
\phi_{k} = \sum_{j=j_{k-1}+1}^{j_k} z_j\varphi_j, \quad k=1,2,\cdots,\ell.
\end{align}
Since $\varphi_j$,
$j=j_{k-1}+1,\dots, j_k$,  are the eigenfunctions
 corresponding to the same eigenvalue $\mu_k$, we deduce from \eqref{equ-54-2} that for $k=1,2,\cdots,\ell$,
\begin{align}\label{equ-54-3}
(-\Delta+V)\phi_{k} = \mu_k \phi_{k} \quad \mbox{over}\;\; E.
\end{align}
Since $E$ is open, acting $(-\Delta+V)$ on both sides of \eqref{equ-54-1}, using \eqref{equ-54-3} and the continuity of $\phi_k$, we obtain that
\begin{align}\label{equ-54-4}
\mu_1\phi_1 + \cdots +\mu_\ell \phi_\ell=0 \quad  \mbox{over}\;\; E.
\end{align}
Similarly, we can also obtain that when $m=0,1,\dots,\ell-1$,
\begin{align}\label{equ-54-5}
\mu_1^m\phi_1 + \cdots +\mu_\ell^m \phi_\ell=0 \quad  \mbox{over}\;\; E.
\end{align}
Since $\mu_1,\cdots,\mu_\ell$ are distinct,     we have
$$
\det \left(
\begin{array}{cccc}
1   &  1   & \cdots &   1     \\
\mu_1 & \mu_2 & \cdots & \mu_\ell\\
\vdots & \vdots & \cdots & \vdots\\
\mu_1^{\ell-1} & \mu_2^{\ell-1} & \cdots & \mu_\ell^{\ell-1}
\end{array}
\right)\neq 0.
$$
(The left hand side is   the Vandermonde determinant.)
This, together with \eqref{equ-54-5}, gives
\begin{align}\label{equ-54-6}
 \phi_1 =  \cdots =\phi_\ell=0 \quad  \mbox{over}\;\;  E.
\end{align}
Now, by  \eqref{equ-54-2} and \eqref{equ-54-6}, we can use the similar method as that used in
the first case to see that
$z_1=z_2=\cdots=z_N=0$, which  contradicts $z\neq \textbf{0}$.

 Hence, we end the proof of   Lemma \ref{lem-invert}.
\end{proof}

We now return to the proof of Theorem \ref{thm-general-2}. Let  $E\subset \R^n$ be a non-empty open subset. According to Lemma \ref{lem-invert},    the matrix $\mathbb{A}$ defined by \eqref{equ-51-4} is invertible.
 Thus, we can define an operator $K: \psi\in L^2(\mathbb{R}^n)\rightarrow K\psi\in L^2(\mathbb{R}^n)$ by
 \begin{align}\label{equ-51-14}
(K\psi)(x):= \rho \left\langle \mathbb{A}^{-1}\left(
\begin{array}{c}
(\psi,\varphi_1)\\
(\psi,\varphi_2)\\
\cdots\\
(\psi,\varphi_N)
\end{array}
\right),
\left(
\begin{array}{c}
 \varphi_1(x) \\
 \varphi_2(x) \\
\cdots\\
 \varphi_N(x)
\end{array}
\right)  \right\rangle_{\R^N}\;\;\mbox{for each}\;\;x\in \mathbb{R}^n,
\end{align}
where $\rho =\lambda_1-1$.
Consider the closed-loop system:
\begin{align}\label{equ-51-heat}
y_t-\Delta y+Vy = \chi_E K y, \quad t>0;\; \qquad y(0,\cdot)\in L^2(\R^n).
\end{align}
We will treat any solution  to \eqref{equ-51-heat} as a function from $[0,\infty)$ to $L^2(\mathbb{R}^n)$,
denoted by $y(t), t\geq 0$.

According to ($\textbf{b}_1$)  in Subsection \ref{concept},
 to get  Theorem \ref{thm-general-2}, it suffices to show the following two assertions:

\vskip 5pt
\noindent {\it Assertion 1} $K$ is a linear bounded  operator  on  $L^2(\R^n)$.

\noindent {\it Assertion 2}
   There is $C>0$ and $\omega>0$ so that each solution $y$ to \eqref{equ-51-heat} satisfies
  $$
  \|y(t)\|_{L^2(\R^n)} \leq Ce^{-\omega t}\|y(0)\|_{L^2(\R^n)}\;\;\mbox{for each}\;\; t\geq 0.
  $$
\vskip 5pt

We first show {\it Assertion 1}. By \eqref{equ-51-14}, it is clear that $K$ is linear.
 Since $\varphi_j$, $j=1,\dots,N$, are  orthonormal, it follows from \eqref{equ-51-14}
 that
\begin{align}\label{equ-54-10}
\|K\psi\|_{L^2(\R^n)} &= |\rho|\left\|\mathbb{A}^{-1}\left(
\begin{array}{c}
(\psi,\varphi_1)\\
(\psi,\varphi_2)\\
\cdots\\
(\psi,\varphi_N)
\end{array}
\right)\right\|_{l^2(\R^N)}\nonumber\\
&\leq |\rho| \|\mathbb{A}^{-1}\|\left( \sum_{j=1}^N| (\psi,\varphi_j)|^2 \right)^{1/2} \\
&\leq |\rho| \|\mathbb{A}^{-1}\|\|\psi\|_{L^2(\R^n)}
\;\;\mbox{for each}\;\;\psi\in L^2(\mathbb{R}^n),  \nonumber
\end{align}
which shows that $K$ is bounded. This leads to {\it Assertion 1}.

We next show {\it Assertion 2}.
 Taking the inner product (in $L^2(\R^n)$) in \eqref{equ-51-heat} with $\varphi_j$, $1\leq j\leq N$, we find
\begin{align}\label{equ-51-15}
&\frac{\d}{\d t} \left( \begin{array}{c}
(y(t),\varphi_1)\\
(y(t),\varphi_2)\\
\cdots\\
(y(t),\varphi_N)
\end{array}
\right) + \left(
\begin{array}{cccc}
\lambda_1 & 0 &\cdots & 0\\
0 & \lambda_2 & \cdots & 0\\
\vdots & \vdots & \cdots &\vdots\\
0 & 0& \cdots &  \lambda_N
\end{array}
\right) \left(
\begin{array}{c}
(y(t),\varphi_1)\\
(y(t),\varphi_2)\\
\cdots\\
(y(t),\varphi_N)
\end{array}
\right)\\
&
= \mathbb{B} \left(
\begin{array}{c}
(y(t),\varphi_1)\\
(y(t),\varphi_2)\\
\cdots\\
(y(t),\varphi_N)
\end{array}
\right),\;\;t\geq 0,
\end{align}
where
$$
\mathbb{B} =
\rho \left(
\begin{array}{cccc}
(\varphi_1,\varphi_1)_E & (\varphi_1,\varphi_2)_E & \cdots & (\varphi_1,\varphi_N)_E\\
(\varphi_2,\varphi_1)_E & (\varphi_2,\varphi_2)_E & \cdots & (\varphi_2,\varphi_N)_E\\
\vdots & \vdots & \cdots & \vdots\\
(\varphi_N,\varphi_1)_E & (\varphi_N,\varphi_2)_E & \cdots & (\varphi_N,\varphi_N)_E
\end{array}
\right) \mathbb{A}^{-1} =
\left(
\begin{array}{cccc}
\rho & 0 &\cdots & 0\\
0 & \rho & \cdots & 0\\
\vdots & \vdots & \cdots &\vdots\\
0 & 0& \cdots &  \rho
\end{array}
\right).
$$
Since $\rho= \lambda_1-1$, it follows from \eqref{equ-51-15} that
\begin{align}\label{equ-51-16}
&\frac{\d}{\d t} \left( \begin{array}{c}
(y(t),\varphi_1)\\
(y(t),\varphi_2)\\
\cdots\\
(y(t),\varphi_N)
\end{array}
\right)\\
&
= \left(
\begin{array}{cccc}
-1 & 0 &\cdots & 0\\
0 & -1+\lambda_1-\lambda_2 & \cdots & 0\\
\vdots & \vdots & \cdots &\vdots\\
0 & 0& \cdots &  -1+\lambda_1-\lambda_N
\end{array}
\right) \left(
\begin{array}{c}
(y(t),\varphi_1)\\
(y(t),\varphi_2)\\
\cdots\\
(y(t),\varphi_N)
\end{array}
\right),\;\;t\geq 0.
\end{align}
Let $P_{\leq N}$ be the projection operator defined by
$$
P_{\leq N}\varphi := \sum_{1\leq j\leq N}(\varphi,\varphi_j) \varphi_j,\quad\;\;\varphi\in L^2(\mathbb{R}^n).
$$
Write
$$
P_{>N}:=I-P_{\leq N},\;\;\mbox{with}\;\;I\;\mbox{the identity operator on}\;\;L^2(\mathbb{R}^n).
$$
Since $\lambda_1\leq \lambda_j$, when $2\leq j\leq N$, we get from \eqref{equ-51-16} that
\begin{align}\label{equ-51-17}
\|P_{\leq N}y(t)\|_{L^2(\R^n)}\leq e^{-t}\|P_{\leq N}y(0)\|_{L^2(\R^n)}, \quad \mbox{when}\;\; t\geq 0.
\end{align}

It remains to estimate   $\| P_{> N} y(t,\cdot)\|_{L^2(\R^n)}$.
To this end,
we take the inner product (in $L^2(\R^n)$)
in  \eqref{equ-51-heat} with ${(y(t),\varphi_j)}\varphi_j$
to see that when  $j\geq N+1$,
\begin{align}\label{equ-54-12}
\frac{1}{2}\frac{\d}{\d t}|(y(t),\varphi_j)|^2 +\lambda_{j}|(y(t),\varphi_j)|^2 =  {(y(t),\varphi_j)}(Ky(t),\varphi_j),\;\;t\geq 0.
\end{align}
Meanwhile, by the Cauchy-Schwartz inequality, we get that when  $j\geq N+1$,
\begin{align}\label{5.31-6-7wang}
|  {(y(t),\varphi_j)}(Ky(t),\varphi_j)|
\leq \frac{\lambda_j}{2}|(y(t),\varphi_j)|^2
+ \frac{1}{2\lambda_j}|(Ky(t),\varphi_j)|^2,\;\;t\geq 0.
\end{align}
Taking the sum in \eqref{equ-54-12} over $j\geq N+1$, using \eqref{5.31-6-7wang}
and
the fact $\lambda_j\geq \lambda_{N+1}$, we infer
\begin{eqnarray}\label{equ-54-13}
 &&\frac{\d}{\d t}\|P_{>N}y(t)\|_{L^2(\R^n)}^2
  +\lambda_{N+1}\|P_{>N}y(t)\|_{L^2(\R^n)}^2\nonumber\\
  &\leq& \frac{1}{\lambda_{N+1}} \|P_{> N}Ky(t)\|^2_{L^2(\R^n)}\nonumber\\
 &\leq& \frac{1}{\lambda_{N+1}}(|\rho|\|\mathbb{A}^{-1}\|)^2\|P_{\leq N}y(t)\|_{L^2(\R^n)}^2,
 \;\;t\geq 0.
\end{eqnarray}
(In  the last step of \eqref{equ-54-13}, we used  \eqref{equ-54-10}.) Applying the Gronwall inequality to \eqref{equ-54-13} and   using   \eqref{equ-51-17}, we deduce  that when $t\geq 0$,
\begin{align}\label{equ-54-14}
\|P_{>N}y(t)\|_{L^2(\R^n)}^2
 & \leq e^{-\lambda_{N+1}t} \|P_{>N}y(0)\|_{L^2(\R^n)}^2 \nonumber \\
& ~ + \frac{1}{\lambda_{N+1}}(|\rho|\|\mathbb{A}^{-1}\|\|P_{\leq N}y(0)\|_{L^2(\R^n)})^2 \int_0^t e^{-\lambda_{N+1}(t-s)}e^{-2s}\d s\nonumber\\
 &\leq e^{-\lambda_{N+1}t} \|P_{>N}y(0)\|_{L^2(\R^n)}^2
  + C_1e^{-C_2t}\|P_{\leq N}y(0)\|^2_{L^2(\R^n)}
\end{align}
for some  $C_1,C_2>0$. By \eqref{equ-51-17} and \eqref{equ-54-14}, we obtain
$$
\|y(t)\|_{L^2(\R^n)}^2\leq C_3e^{-C_4t}\|y(0)\|^2_{L^2(\R^n)}, \quad t\geq 0
$$
for some $C_3,C_4>0$. This leads to {\it Assertion 2}.

Thus, we complete the proof of Theorem \ref{thm-general-2}.

\section*{Acknowledgment}
The authors would like to thank  Dr. Can Zhang for very useful discussions on observability of fractional heat equations.

S. Huang was supported by the National Natural Science Foundation of China under grants 11801188
and 11971188.

G. Wang was partially supported by the National Natural Science Foundation of China under grants 11971022 and 11926337.

M. Wang was supported by the National Natural Science Foundation of China under grant No. 11701535.
%%%%%%%%%%%%%%%%%%%%%%%%%%%%%%%%%%%%%%%%%%%%%%%%

%%%%%%%%%%%%%%%%%%%%%%%%%%%%%%%%%%%%%%%%%%%%%%%%%
\end{document}